\documentclass[a4paper,oneside,11pt]{article}
\usepackage{amsmath}
\usepackage{indentfirst}

\usepackage{amssymb}
\usepackage{abstract}
\usepackage{titlesec}
\usepackage{graphicx}
\usepackage{fancyhdr}
\usepackage{latexsym}
\usepackage{mathrsfs}
\usepackage{booktabs}
\usepackage{mathrsfs}
\usepackage{mathtools}
\usepackage{amsthm}
\usepackage{wasysym}
\usepackage{cite}
\usepackage{color}
\usepackage{todonotes}
\usepackage{geometry}
\numberwithin{equation}{section}
\newtheorem{definition}{Definition}[section]
\newtheorem{theorem}{Theorem}[section]
\newtheorem{proposition}{Proposition}[section]
\newtheorem{lemma}{Lemma}[section]
\newtheorem{remark}{Remark}[section]
\newenvironment{solution}{\begin{proof}[\bf Proof of Proposition \ref{p1}]}{\end{proof}}

\newcommand{\curl}{{\rm curl} }
\newcommand{\la}{\label}
\renewcommand{\div}{ {\rm div }  }

\newcommand{\na}{\nabla }
\newcommand{\bt}{\begin{theorem}}
\newcommand{\bl}{\begin{lemma}}
\newcommand{\el}{\end{lemma}}
\newcommand{\et}{\end{theorem}}

\newcommand{\bn}{\begin{eqnarray}}
\newcommand{\en}{\end{eqnarray}}
\newcommand{\bnn}{\begin{eqnarray*}}
\newcommand{\enn}{\end{eqnarray*}}

\newcommand{\bnnn}{\begin{eqnarray*}}
\newcommand{\ennn}{\end{eqnarray*}}
\newcommand{\ben}{\begin{enumerate}}
\newcommand{\een}{\end{enumerate}}
\newcommand{\ba}{\begin{aligned}}
\newcommand{\ea}{\end{aligned}}
\newcommand{\be}{\begin{equation}}
\newcommand{\ee}{\end{equation}}

\title{Local well-posedness of strong solutions to the  compressible Navier-Stokes equations with degenerate viscosities and far field vacuum in 3D exterior domains}
\author{Jiaxu L{\small I}$^{a}$, Boqiang L{\small \"U}$^{b}$, Bing Y{\small UAN}$^{b}$, \thanks{email: jiaxvlee@gmail.com (J.X. Li), lvbq86@163.com (B.Q.   L\"u),  bingyuan@email.ncu.edu.cn (B. Yuan)} \\
{\normalsize a. School of Mathematical Sciences,}\\
{\normalsize Shenzhen University, Shenzhen 518061, P. R. China;}\\
{\normalsize b. School of Mathematics and Computer Sciences}\\ {\normalsize \& Institute of Mathematics and Interdisciplinary Sciences,}\\ {\normalsize  Nanchang University, Nanchang 330031, P. R. China;}
}
\date{}

\begin{document}
\maketitle
\begin{abstract}
The isentropic compressible Navier-Stokes system subject to the Navier-slip boundary conditions is considered in a general three-dimensional exterior domain. For the density approaches far-field vacuum initially and the viscosities are power functions of the density ($\rho^\delta$ with $0<\delta<1$), the local well-posedness of strong solutions is established in this paper. In particular, the method we adopt can not only simultaneously handle the difficulties caused by boundary terms and far-field vacuum, but also make the selection of $\delta$ independent of the gas coefficient $\gamma$. 
%Moreover, we obtain a blowup criterion of the strong solution. 
\end{abstract}

\textbf{Keywords:} Compressible Navier-Stokes equations, Degenerate viscosity, Far-field vacuum, Exterior domain

\textbf{MSC 2020:} 76N06, 76N10, 35Q35

\section{Introduction}
We consider  the three-dimensional isentropic compressible Navier-Stokes equations as follows:
%in an exterior domain $\Omega\subset\mathbb{R}^3$:
\begin{equation}\label{1}
\left\{ \begin{array}{l}
\rho_t+\mathrm{div}(\rho u)=0,\\
(\rho u)_t+\mathrm{div}(\rho u\otimes u)+\nabla P=\mathrm{div}\mathbb{T}. \\
\end{array} \right.
\end{equation}
Here, $t\geq0$, $x=(x_1,x_2,x_3)\in\Omega\subset\mathbb{R}^3$ are time and space variables, respectively. $\rho=\rho(x,t)$, $u=(u_1(x,t),u_2(x,t),u_3(x,t))$, and $P(\rho)=a\rho^\gamma (a>0, \gamma>1)$ are the unknown fluid density, velocity, and pressure, respectively.
% The pressure $P$ is given by
%\begin{equation}
%P=a\rho^\gamma,\ (a>0, \gamma>1),
%\end{equation}
%where $a$ is an entropy constant and $\gamma$ is the adiabatic exponent.
$\mathbb{T}$ denotes the viscous stress tensor with the form
\begin{equation}
\mathbb{T}=2\mu(\rho)\mathcal{D}(u)+\lambda(\rho)\mathrm{div}u\mathbb{I}_3,
\end{equation}
where $\mathcal{D}(u)=\frac12\big[\nabla u+(\nabla u)^T\big]$ is the deformation tensor, $\mathbb{I}_3$ is the $3\times 3$ identity matrix. The shear viscosity $\mu(\rho)$ and the bulk one $\lambda(\rho)+\frac{2}{3}\mu(\rho)$ satisfy the following hypothesis
\begin{equation}\label{vis-c}
\mu(\rho)=\mu\rho^\delta,\ \lambda(\rho)=\lambda\rho^\delta,
\end{equation}
for some constant $\delta\geq 0$, $\mu$ and $\lambda$ are both constants satisfying
\begin{equation}\label{5}
\mu>0,\ 2\mu+3\lambda\geq0.
\end{equation}

Let $D$ be a simply connected bounded domain in $\mathbb{R}^3$ with smooth boundary, and $\Omega=\mathbb{R}^3\backslash\bar{D}$ be the exterior domain. In this paper, we study the initial-boundary value problem (IBVP) of $\eqref{1}$ in $\Omega$ with the initial condition
\begin{align}\label{6}
\rho(x,0)=\rho_0(x),\,\,\,\ u(x,0)=u_0(x), \,\,\,x\in \Omega,
\end{align}
and the Navier-type slip boundary condition (see \cite{Cai2023, Navier1827})
\begin{equation}\label{ch1}
u\cdot n=0, \,\,\, \mbox{curl}u\times n=-A(x)u,\,\,\, \mbox{on}\,\,\, \partial\Omega,
\end{equation}
%where $A = A(x)$ is a $3\times 3$ symmetric matrix defined on $\partial\Omega$.
where $n=(n^1,n^2,n^{3})$ is the unit normal vector to the boundary $\partial\Omega$ pointing outside $\Omega$,
$A$ is a $3\times3$ symmetric matrix defined on $\partial\Omega$.
Also, we consider the following far-field behavior
\begin{align}\label{7}
(\rho,u)\rightarrow(0,0),\,\,\,as\,\,\, |x|\rightarrow\infty.
\end{align}

There is a huge literature on the well-posedness of solutions for the multidimensional isentropic compressible Navier-Stokes system. For the case of constant viscosities ($\delta=0$ in \eqref{vis-c}), the local existence and uniqueness of  classical (strong)  solutions is  established in \cite{Nash1962,Serrin1959,Tani19774} (without vacuum)  and  \cite{3,4,5,13,Straskraba1993,Huang2020} (with vacuum).
The global classical solutions were first obtained by Matsumura-Nishida \cite{matsumura1980initial} for initial data close to a non-vacuum equilibrium in $H^3$. Later, Hoff \cite{hoff1995,hoff1997} studied the global weak solutions with strictly positive initial density and temperature for discontinuous initial data. For the existence of solutions for large data, the major breakthrough is due to Lions \cite{Lions1998-2}, where he obtained the global existence of weak solutions with finite energy, under the condition that the adiabatic exponent $\gamma\ge\frac{9}{5}$ (3D problem), see also Feireisl \cite{Feireisl2001} for $\gamma>\frac{3}{2}$. Recently, Huang-Li-Xin \cite{HuangLiXin2012} obtained the global classical solutions to the Cauchy problem with small energy but possibly large oscillations.

For density-dependent viscosities ($\delta>0$ in \eqref{vis-c}), 
%instead of the uniform elliptic structure, 
the momentum equation $\eqref{1}_2$ becomes a double degenerate parabolic equation when vacuum appears, which brings new difficulties in studying the well-posedness of solutions to the system \eqref{1}.
Recently, Li-Xin \cite{9} and Vasseur-Yu \cite{Vasseur2016} independently investigated the global weak solutions for compressible Navier-Stokes systems adhere to the Bresch-Desjardins relation \cite{bresch2004some}.
For the case $\delta = 1$, Li-Pan-Zhu \cite{7} obtained the existence of 2D local classical solution with far field vacuum, which also applies to the 2D shallow water equations. They in \cite{6} also obtained the 3D local classical solutions when the $\delta > 1$, which has been extended to be a global one by Xin-Zhu \cite{MR4340224} for a class of smooth initial data that are of small density but possibly large velocities in some homogeneous Sobolev spaces. In the case of $\delta\in(0,1)$, under some limitation on $\delta$ and $\gamma$ as follows:
\begin{equation}\label{de-ga}\delta\rightarrow1,~~~\text{when}~~\gamma\rightarrow1,
\end{equation}
Xin-Zhu \cite{zhuxin2021}  established the local classical solution with a far-field vacuum in 3D total space. Later, through a more meticulous processing, Li-Li \cite{Lili2025}  expanded the range of coefficient values in Xin-Zhu \cite{zhuxin2021} and 
also proved the 3D local classical solution with a far-field vacuum, in which some more physical scenarios, such as Maxwellian molecules, are therefore included.

When it comes to the IBVP in  exterior domains, the present well-posedness theories mainly focus on the constant viscosity coefficient. Cai-Li-L\"u \cite{Cai2112} and Li-Li-L\"u \cite{Lili2022} respectively establish the global existence of classical solutions to the isentropic and full compressible Navier-Stokes equations, which are of small energy but possibly large oscillations for the IBVP with slip boundary condition in exterior domains.
For the case with density-dependent viscosities \eqref{vis-c}, the degenerate viscosity coefficient will bring new difficulties to the boundary estimations. There are few works on the IBVP to system \eqref{1} in exterior domains with vacuum and the general Navier-slip boundary condition, even for local solutions. Very recently, using the similar arguments as \cite{zhuxin2021},   Liu-Zhong \cite{liu2025ns} proved the local classical solution with a far-field vacuum in 3D  exterior domain under the condition  \eqref{de-ga}.

Therefore, this paper aims to develop new technical methods to extend the Cauchy problem to more general boundary value problems, while making the selection of $\delta$ independent of the gas coefficient $\gamma$ to encompass a wider range of physical scenarios. 
Motivated by Li-Li \cite{Lili2025}, we first reformulate \eqref{1} as
\begin{equation}\label{e19}
\left\{ \begin{array}{l}
  \rho_t+\mathrm{div}(\rho u)=0,\\
  \rho^{1-\delta}u_t+\rho^{1-\delta}u\cdot\nabla u-\mathcal{L}u+\frac{a\gamma}{\gamma-\delta}\nabla\rho^{\gamma-\delta}=\delta\nabla\log\rho\cdot \mathcal{S}(u),
       \end{array} \right.
\end{equation}
where
\begin{equation}\label{10}
\mathcal{L}u=\mu\Delta u+(\mu+\lambda)\nabla \mathrm{div}u,\quad\mathcal{S}(u)=2\mu\mathcal{D}(u)+\lambda \mathrm{div}u\mathbb{I}_3.%\quad\mathcal{D}(u)=\nabla u+(\nabla u)^T.
\end{equation}

The strong solutions of IBVP \eqref{1}-\eqref{7} considered in this paper is defined   as follows:
\begin{definition}
Let $T>0$ be a finite constant. A solution $(\rho,u)$ to the system \eqref{1}-\eqref{7} is called a strong solution if all the derivatives involved in \eqref{1} for $(\rho,u)$ are regular distributions, and Eqs.\eqref{1} hold almost everywhere in $[0, T]\times \Omega$.
\end{definition}

Before stating the main results, we first explain the notations and conventions
used throughout this paper. For a positive integer $k$ and $p\geq1$, we denote the standard Lebesgue and Sobolev spaces as follows:
$$
\begin{gathered}
\|f\|_{L^p}=\|f\|_{L^p(\Omega)},\ \|f\|_{W^{k,p}}=\|f\|_{W^{k,p}(\Omega)},\ \|f\|_{H^k}=\|f\|_{W^{k,2}(\Omega)},\ \|f\|_{L^{p}}=\|f\|_{W^{0,p}(\Omega)},\\\
D^{k,p}=\{f\in L^1_{loc}(\Omega):|f|_{D^{k,p}}=\|\nabla^kf\|_{L^p}<\infty\},\ D^k=D^{k,2},\\
D_0^{1}(\Omega)=\{f\in L^{6}(\Omega):\|\nabla f\|_{L^2}<\infty\},\ \|f\|_{X\cap Y}=\|f\|_X+\|f\|_Y.
\end{gathered}
$$
Let $B_R=\{x\in\mathbb{R}^3||x|<R\}$, we define  $$\Omega_R\triangleq \Omega\cap B_R$$ 
where  $R>2R_0+1$ with $R_0$ is  chosen to be sufficiently large such that $\bar D\subset B_{R_0}$. In particular, we denote $\Omega_0\triangleq \Omega\cap B_{2R_0}$. 

The main result of this paper is the following Theorem \ref{T1} concerning the local existence of strong solution.

\begin{theorem}\label{T1}
%The definition of $\Omega_1$ and
For parameters $(\gamma,\delta)$ satisfy
\begin{equation}
%\Omega_1\triangleq\{x\in\Omega|dist(x,\partial \Omega)\leq 1\},\quad
\gamma>1,\ 0<\delta<1. \label{115}
\end{equation}
If the initial data $(\rho_0, u_0)$ satisfies
\begin{equation}
\begin{split}\label{1117}
\rho_0^\frac{1-\delta}{2}u_0\in L^2,\ u_0\in D^{1}\cap D^2,\
\rho_0^{\gamma-\frac{1+\delta}{2}}\in D_0^{1}\cap D^2,\ \nabla\rho_0^\frac{\delta-1}{2}\in D_0^{1},%\  \rho_0^\frac{\delta-1}{2}\in L^6(\Omega_1),
\end{split}
\end{equation}
and the compatibility condition:
\begin{equation}\label{118}
\mathcal{L}(u_0)=\rho_0^\frac{1-\delta}{2}g
\end{equation}
for some $g\in L^2$. Then there exist a positive time $T_0>0$ such that the problem  \eqref{1}-\eqref{7} has a unique strong solution $(\rho, u)$ on $[0,T_0]\times \Omega$ satisfying 
\begin{equation}\label{e116}
\left\{ \begin{array}{l}
\rho^{\gamma-\frac{1+\delta}{2}}\in L^\infty([0,T_0];D_0^1\cap D^2),\\
\nabla\rho^\frac{\delta-1}{2}\in L^\infty([0,T_0];D_0^{1}),
\quad\rho^\frac{\delta-1}{2}\in L^\infty([0,T_0];L^6(\Omega_0)),\\
 u\in L^\infty([0,T_0];D_0^1\cap D^2)\cap L^2([0,T_0];D_0^1\cap D^3),\\
\rho^\frac{1-\delta}{2}u_t \in L^\infty([0,T_0];L^2),\quad u_t\in L^2([0,T_0];D_0^1),\\
\rho^\frac{\delta-1}{2}Lu,\ \rho^\frac{\delta-1}{2}\nabla\mathrm{div}u\in L^2([0,T_0];H^1).
    \end{array} \right.
\end{equation}
\end{theorem}

\begin{remark}
The conditions \eqref{115}-\eqref{1117} in Theorem \ref{T1} identify a class of admissible
initial data that makes the problem \eqref{1}-\eqref{7} solvable, which are satisfied by, for example,
\begin{equation}
\rho_0(x)=\frac{1}{1+|x|^{2\alpha}},\ u_0(x)\in C_0^2(\Omega),~\text{with}~~\frac{1}{2(2\gamma-1-\delta)}<\alpha<\frac{1}{2(1-\delta)}.
\end{equation}
In particular, the range of $\delta$ is independent of $\gamma$. This illustrates that Theorem \ref{T1} is applicable to all $\gamma>1$ and $\delta\in(0,1)$, which  is different from the results of Xin-Zhu \cite{zhuxin2021} and Liu-Zhong \cite{liu2025ns}, in which the limitation \eqref{de-ga} is needed.  Furthermore,   the only compatibility condition  required in Theorem \ref{T1} is \eqref{118}, which removed the other two  additional restrictive constraints in   \cite{zhuxin2021,liu2025ns}.  
Therefore, our Theorem \ref{T1} holds true for  more general initial data.  
Moreover, if $2\gamma\leq\delta+\frac{4}{3}$, then $\rho_0\in L^1(\Omega)$, and the density $\rho$ constructed in Theorem \ref{T1} may have finite total mass.
\end{remark}

\begin{remark}
In fact, based on the choice of $\rho_0$ (see \eqref{1117}) and Lemma \ref{087}, we can obtain that $\rho_0$ has a positive lower bound on $\Omega_0$, and therefore it has 
\begin{equation}\label{260114}
    \rho_0^\frac{\delta-1}{2}\in L^6(\Omega_0). 
\end{equation}Theorem \ref{T1} gives a new priori estimate 
%Through standard calculations, we can obtain 
$\rho^\frac{\delta-1}{2}\in L^\infty([0,T_0];L^6(\Omega_0))$, which combining with $\nabla\rho^\frac{\delta-1}{2}\in L^\infty([0,T_0];D_0^1)$ gives that $\rho$ has a positive lower bound near the boundary $\partial\Omega$. Roughly speaking, the density  retains a positive lower bound on $\Omega_0$ as initial one. This new observation  plays a crucial role in handing the boundary estimation of $u$. It should be mentioned here that this is a significant difference  compared to the Cauchy problem in Li-Li \cite{Lili2025} and  IBVP in  exterior domains Liu-Zhong \cite{liu2025ns}.
%, one in literature \cite{zhuxin2021, liu2025ns}. (Formally speaking, our approach is to divide the equation \eqref{1}$_2$ by $\rho^\delta$ while they divide it by $\rho$.)
\end{remark}

\begin{remark}
For the time continuity, similar to \cite[Remark 1.2]{Lili2025}, we can also deduce from \eqref{e116} and the classical Sobolev embedding results that
\begin{gather}
    \rho^{\gamma-\frac{1+\delta}2}\in C([0, T_0];D_0^1\cap D^2),\ \nabla\rho^{\frac{\delta-1}{2}}\in C([0, T_0];D_0^1),
    \\ 
    u\in C([0,T_0];D_0^1\cap D^2),\ 
    \rho^\frac{1-\delta}2u,\ \rho^\frac{1-\delta}2u_t\in C([0,T_0];L^2).
\end{gather}
\end{remark}

\begin{remark}
Theorem \ref{T1} shows that
\begin{align}
\rho^{\gamma-\frac{1+\delta}{2}}\in L^\infty([0,T_0];D_0^1\cap D^2),\quad \nabla\rho^\frac{\delta-1}{2}\in L^\infty([0,T_0];D_0^{1}),
\end{align}
which together with Lemma \ref{087} implies that a vacuum can only occur at infinity.
\end{remark}

%\begin{remark}
%Compared with \cite{liu2025ns, zhuxin2021} requiring that the %initial data belong to $H^3$, %to reduce the regularity requirements on the initial values 
%we only need the initial data belong to $H^2$ (see %\eqref{1117} for details).
%Moreover, we only need one compatibility condition, while they %need three. Therefore, our results have lower limitations on %the initial values.
%\end{remark}
\begin{remark}
    In \cite{9} and \cite{Vasseur2016}, the viscosity coefficients $\mu(\rho)$ and $\lambda(\rho)$ are required to satisfy B-D relation:
    \begin{equation}
        \lambda(\rho)=2(\mu'(\rho)\rho-\mu(\rho)),
    \end{equation}
    which offers an estimate $\mu'(\rho)\nabla\sqrt\rho\in L^\infty([0,T];L^2(\mathbb{R}^3))$ provided that
    $\mu'(\rho_0)\nabla\sqrt\rho_0\in L^2(\mathbb{R}^3)$. The additional restriction is not necessary in this paper. 
\end{remark}

\begin{remark}
The general Navier-type slip condition derived in \cite{Navier1827} is stated as follows:
\begin{align}\label{ch2}
u\cdot n=0,\,\,\,\ \big(2\mathcal{D}(u) n+\vartheta u\big)_{tan}=0, \,\,\,\text{on} \,\,\,\partial\Omega,%\,\,\,\text{or}\,\,\ u=0 \,\,\,\text{on} \,\,\,\partial\Omega,
\end{align}
%where $\mathcal{D}(u) = \big(\nabla u + (\nabla u)^{T}\big)/2$ is the shear stress,
where $\vartheta$ is a scalar friction function which measures the tendency of the fluid to slip on the boundary, and the symbol $v_{tan}$ represents the projection of tangent plane of the vector $v$ on $\partial\Omega$. In fact, the boundary condition \eqref{ch2} is equivalent to \eqref{ch1} in the sense of the distribution (see \cite{{Cai2023}} for details).
\end{remark}

Finally, to answer the question of whether the local solution constructed in Theorem 1.1 can be extended globally in time, similarly to Huang-Li \cite{14}, we have the following blow-up criterion:
\begin{theorem}\label{T2}
Let $(\rho, u)$ be a strong solution to the IBVP \eqref{1}-\eqref{7} satisfying \eqref{e116}. Assume that the initial data
$(\rho_0, u_0)$ satisfies \eqref{115}-\eqref{118} and $\rho_0^{\gamma-\frac{1+\delta}{2}}\in L^{6\alpha}$, where
$\max\{\frac{1-\delta}{2\gamma-1-\delta},\frac13\}<\alpha<1$. If $T^\ast<\infty$ is the maximal time of existence, then
\begin{equation}\label{blowup1}
\lim_{T\rightarrow T^\ast}\left(\|\mathcal{D}(u)\|_{L^1(0,T;L^\infty)}+\|\rho^\frac{\delta-1}{2}\|_{L^\infty(0,T;L^6(\Omega_0))} +\|\nabla\rho^\frac{\delta-1}{2}\|_{L^\infty(0,T;L^6\cap D^{1,2})}\right)=+\infty.
\end{equation}
%where $\mathcal{D}(u)$ is the deformation tensor defined by \eqref{10}.
\end{theorem}

\begin{remark}
For the Navier-type slip boundary condition \eqref{ch1}, since negative powers of $\rho$  may occur when dealing with boundary estimation of $u$, one should make sure that $\rho$ has a positive lower bound near the boundary $\partial\Omega$ (see \eqref{trace14}, \eqref{trace31}) and thus the second term in  \eqref{blowup1} is needed. 
In particular, if $A$ is a semi-positive definite matrix, the  corresponding  boundary   estimates are trivial,  and the blow-up criterion is simplified to
\begin{equation}
\lim_{T\rightarrow T^\ast}\left(\|\mathcal{D}(u)\|_{L^1(0,T;L^\infty)} +\|\nabla\rho^\frac{\delta-1}{2}\|_{L^\infty(0,T;L^6\cap D^{1,2})}\right)=+\infty.
\end{equation}
\end{remark}

We now make some comments on the analysis in this paper. Motivated by Li-Li \cite{Lili2025}, we  divide the equation \eqref{1}$_2$ by $\rho^\delta$ and then investigate the resulting equations \eqref{e19}  to establish the strong solutions of  system \eqref{1}. 
In contrast to the Cauchy problem, the boundary conditions \eqref{ch1} will present fundamental challenges, which require the development of new technical approaches.  

First, we consider the system \eqref{e19} in  annular region $\Omega_R$ which is  utilized   %the cut-off function to combine a priori estimates of the elliptic problem of general boundary conditions in a bounded domain and the $B_R$ to obtain a priori estimates necessary in 
to approach the exterior domain $\Omega$, and then establish the necessary a priori estimates independent of $R$. Note that the  first Betti number of  $\Omega_R$ is not zero, 
the div-curl Lemma \ref{l22} is as follows
%between exterior domains and bounded domains, especially when the first Betti number is not zero, like $\Omega_R$. 
\begin{equation}\label{0117}
    \|\nabla u\|_{L^2(\Omega_R)}\leq C\bigr(\|\div u\|_{L^2(\Omega_R)}+\|\curl u\|_{L^2(\Omega_R)}+\|u\|_{L^2(\Omega_R)}\bigr),
\end{equation}
which is completely different from the case of simply connected domains. Due to the degeneracy of the density at far fields,  the norm $\|u\|_{L^2(\Omega_R)}$ in the right hand of \eqref{0117} lacks a uniform estimate  with respect to $R$ in our framework.  
Fortunately, we can use the cutoff function method to localize $\|u\|_{L^2(\Omega_R)}$ 
%and shift $\|u\|_{L^2(\Omega_R)}$ 
to $\|u\|_{L^2(\Omega_0)}$, which could be controlled by the new  observation of that  the density $\rho$ has the lower bound near the boundary, see  \eqref{dsxc} for details. Indeed, this key observation is also crucial for   dealing with the boundary terms in deriving the a prior estimation of $u$, in which the sign of the boundary terms is undefined when the  matrix $A$ is not necessarily positive definite, see \eqref{trace3}-\eqref{rho_inf}. 

%For low-order estimations, it is sufficient for us to directly apply the trace theorem (see \eqref{trace1}, \eqref{trace2}). For higher-order estimations, we can first apply the trace theorem for processing, combine the estimation of $u_t$ with $\rho^\frac{1-\delta}{2}$ weights and the positive lower bound of $\rho$ near the boundary, and then perform appropriate operations to control it (see \eqref{trace3}-\eqref{rho_inf}).

%Similarly, when dealing with boundary terms in the prior estimation of $u_t$, since matrix $A$ is not necessarily positive definite, the sign of boundary term is unknown. Considering that $\rho$ has a positive lower bound near the boundary $\partial\Omega$, combined with the estimation of $u_t$ with $\rho^\frac{1-\delta}2$ weights, we can directly use the trace theorem to control it (see \eqref{trace3}-\eqref{rho_inf} for details).

%Next, since matrix $A$ in the boundary conditions may not necessarily be semi-positive definite, all boundary terms generated by the prior estimation need to be processed and estimated by norms of $u$. However, parts of the estimates of u carry $\rho$ weights. Fortunately, $\rho$ has a positive lower bound near the boundary (see Lemma \ref{087}), which is sufficient for us to complete the proof of the Lemma \ref{L3.1}-\ref{L3.3}.

Next, as mentioned in \cite{Lili2025},  it's important to  estimate the term  $\nabla\rho^\frac{\delta-1}{2}$, which could be dominated by $w\triangleq\rho^\frac{\delta-1}{2}\nabla\div u$ due to the mass equation. However, the methods used in Li-Li \cite{Lili2025} are not apply to our problem, where the essential difficulties arising from the annular-type domains  $\Omega_R$. More precisely, one cannot derive the desired estimates on $\nabla w$ via the standard elliptic estimates for a Neumann boundary-value problem (see \cite[Lemma 2.4]{Lili2025}) in $\Omega_R$. To address these issues, we need some new  ideas to deduce the crucial estimates of $\|\nabla w\|_{L^2(\Omega_R)}$, see Lemma \ref{L3.6} for details. On the one hand, we deduce form  the equation \eqref{e19}$_2$ that $w$ can be rewritten as 
\begin{align}\label{260122}
(2\mu+\lambda)w=\rho^\frac{\delta-1}{2}\mathcal{L}u+\mu\rho^\frac{\delta-1}{2}\curl\curl u,
\end{align}
where $\rho^\frac{\delta-1}{2}\mathcal{L}u$ satisfies equation \eqref{zsa1} and thus could be estimated directly. On the other hand, with the help of Lemma \ref{l22}, one can deal with the second term in \eqref{260122} as follows 
\begin{align}
&\|\nabla\left(\rho^\frac{\delta-1}{2}\curl\curl u\right)\|_{L^2(\Omega_R)}^2\nonumber\\
&\leq C\|\div\left(\rho^\frac{\delta-1}{2}\curl\curl u\right)\|_{L^2(\Omega_R)}^2+C\|\curl\left(\rho^\frac{\delta-1}{2}\curl\curl u\right)\|_{L^2(\Omega_R)}^2\nonumber\\
&\quad+C\|\rho^\frac{\delta-1}{2}\curl\curl u\cdot n\|_{H^{1/2}(\partial\Omega)}^2+C\|\rho^\frac{\delta-1}{2}\curl\curl u\|_{L^2(\Omega_R)}^2\\
&\leq C\|\nabla\rho^\frac{\delta-1}{2}\|_{L^6(\Omega_R)}^2\|\nabla^2 u\|_{L^3(\Omega_R)}^2+C\|\curl\left(\rho^\frac{\delta-1}{2}\curl\curl u\right)\|_{L^2(\Omega_R)}^2\nonumber\\
&\quad+C\mathcal{E}^\beta\|\curl\curl u\|_{H^1(\Omega_R)}^2+C\|\rho^\frac{\delta-1}{2}\curl\curl u\|_{L^2(\Omega_R)}^2,\nonumber
%&\leq C\mathcal{E}^\beta\bigl(1+\|\nabla u_t\|_{L^2}^2\bigr)+C\|\rho^\frac{\delta-1}{2}\mathcal{L}u\|_{L^2}^2.\nonumber
\end{align}
where the term concerning  $\curl\left(\rho^\frac{\delta-1}{2}\curl\curl u\right)$ can be estimated from \eqref{zsa2}. 

Then,  combining all these ideas stated above with those
due to \cite{3,4,5,13,Lili2025}, we derive some desired a priori estimates on the solutions. Notice that  all these bounds are independent of both the radius  $R$ and the lower bound of the initial density, one can establish the local existence and uniqueness of strong solutions to the system.

Finally,  under an additional initial condition of that $\rho_0^{\gamma-\frac{1+\delta}{2}}\in L^{6\alpha}$, we can deal with the boundary terms in deducing  the estimations of pressure term (see \eqref{4.8}-\eqref{ygt1} for details) and thus obtain the  blow-up criterion in Theorem \ref{T2}.  

The rest of the paper is organized as follows: In Section 2, we collect some elementary facts and inequalities that will be used in later analysis. Section 3 is devoted to the a priori estimates required to establish the local existence and uniqueness of strong solutions. Then the main result, Theorem 1.1, is proved in Section 4. Finally, we give a blow-up criterion for the strong solutions in Section 5.

\section{Preliminaries}
In this section, some useful known facts and inequalities are listed. The first is the well-known Gagliardo-Nirenberg inequality \cite{10}.
\begin{lemma}\label{G-N}
Assume that $U\subset \mathbb{R}^3$ is a bounded Lipschitz domain, for $p\in[2,6]$, $q\in(1,\infty)$, and $r\in(3,\infty)$, there exist generic constants $C$,~$C_1$,~$C_2$ which depend on $p, q, r,$ and the Lipschitz character of $U$, such that for any
$f\in H^1(U),\ and\ g\in L^q(U)\cap D^{1,r}(U),$
one has
\begin{equation}
\|f\|_{L^p(U)}\leq C\|f\|^{(6-p)/2p}_{L^2(U)}\|\nabla f\|^{(3p-6)/2p}_{L^2(U)}+C_1\|f\|_{L^2(U)},
\end{equation}
\begin{equation}
\|g\|_{C(\bar{U})}\leq C\|g\|^{q(r-3)/(3r+q(r-3))}_{L^q(U)}\|\nabla g\|^{3r/(3r+q(r-3))}_{L^r(U)}+C_2\|g\|_{L^q(U)}.
\end{equation}
Moreover, if $f\cdot n|_{\partial U} = 0$, we can choose $C_1 = 0$. Similarly, the constant $C_2 = 0$ provided $g\cdot n|_{\partial U} = 0$.
\end{lemma}

The following lemma, whose proof is established in \cite{Aramaki2014Lp, Cai2023},   plays a crucial role in our analysis.  
\begin{lemma}\label{l22}
For any $u\in W^{1,p}(\Omega_R)$ with  $p\in(1,\infty)$, there exists some positive constant $C$ depending only on $p$ and $R_0$ such that
\begin{itemize}
    \item When $u\cdot n=0$ on $\partial\Omega_R$, one has
\begin{align}\label{uwkq}
\|\nabla u\|_{L^{p}(\Omega_R)}\leq C\bigr(\|\div u\|_{L^{p}(\Omega_R)}+\|\curl u\|_{L^{p}(\Omega_R)}+\|u\|_{L^p(\Omega_{0})}\bigr),
\end{align}
\item when $u\cdot n \neq 0$ on $\partial\Omega_R$, one has
\begin{align}\label{uwkq1}
\|\nabla u\|_{L^{p}(\Omega_R)}\leq C\bigr(\|\div u\|_{L^{p}(\Omega_R)}+\|\curl u\|_{L^{p}(\Omega_R)}+\|u\|_{L^{p}(\Omega_0)}+\|u\cdot n\|_{W^{1-1/p,p}(\partial\Omega_R)}\bigr).
\end{align}
\end{itemize} 
\end{lemma}
\begin{proof}
First, we introduce a cut-off function $\eta(x)\in C_c^\infty(B_{2R_0})$ satisfying 
\begin{equation}
    \label{rro1} \eta(x)=\begin{cases}1,~~~~~~~~  &|x|\leq R_0\\
    \eta(x)\in (0,1), &R_0<|x|<
    \frac{3}{2}R_0,\\ 
    0, &|x|\geq \frac{3}{2}R_0,
    \end{cases}
\end{equation}
and $|\nabla \eta(x)|<C(R_0)$. Thus, it's easy to find that $\eta u$ is a funtion defined on $\Omega_{0}$ and satisfies $\eta u\cdot n=0$ on $\partial\Omega_{0}$.  Consequently,  one obtains after using \cite[Lemma 2.5]{Cai2023} that
\begin{equation}
\begin{split}\label{bytq1}
\|\nabla(\eta u)\|_{L^p(\Omega_R)}&=\|\nabla(\eta u)\|_{L^p(\Omega_{0})}\\
&\leq C\|\curl(\eta u)\|_{L^p(\Omega_{0})}+C\|\div(\eta u)\|_{L^p(\Omega_{0})}+C\|\eta u\|_{L^p(\Omega_{0})}\\
&\leq C\|\eta \curl u\|_{L^p(\Omega_{0})}+C\|\nabla\eta\times u\|_{L^p(\Omega_{0})}+C\|\eta \div u\|_{L^p(\Omega_{0})}\\
&\quad+C\|\nabla\eta\cdot u\|_{L^p(\Omega_{0})}+C\|\eta u\|_{L^p(\Omega_{0})}\\
&\leq C\|\curl u\|_{L^p(\Omega_R)}+C\|\div u\|_{L^p(\Omega_R)}+C\|u\|_{L^p(\Omega_{0})}.
\end{split}
\end{equation}

Similarly, it holds that $(1-\eta) u$ is a function defined on $B_R$ and satisfies $(1-\eta) u\cdot n=0$ on $\partial B_R$. Then, one deduces  from \cite[Lemma 2.2]{Lili2025} that
\begin{equation}\notag
\begin{split}
\|\nabla\big((1-\eta) u\big)\|_{L^p(\Omega_R)}&=\|\nabla\big((1-\eta) u\big)\|_{L^p(B_R)}\\
&\leq C\|\curl\big((1-\eta) u\big)\|_{L^p(B_R)}+C\|\div\big((1-\eta) u\big)\|_{L^p(B_R)}\\
&\leq C\|(1-\eta) \curl u\|_{L^p(B_R)}+C\|\nabla\eta\times u\|_{L^p(B_R)}\\
&\quad+C\|(1-\eta) \div u\|_{L^p(B_R)}+C\|\nabla\eta\cdot u\|_{L^p(B_R)}\\
&\leq C\|\curl u\|_{L^p(\Omega_R)}+C\|\div u\|_{L^p(\Omega_R)}+C\|u\|_{L^p(\Omega_{0})},
\end{split}
\end{equation}
which together with \eqref{bytq1} yields \eqref{uwkq}.

With the similar arguments as those in  \cite{Aramaki2014Lp},  we can also prove \eqref{uwkq1} and complete the proof of Lemma \ref{l22}.
\end{proof}

%The following $L^p$-bound for two elliptic systems,
Next, considering the Lam\'e's system
\begin{equation}\label{2.3}
\begin{cases}
		\mathcal{L}u=f,~~~~~&x\in \Omega_R,\\
		u\cdot n = 0,~\mathrm{curl}u\times n=-Au,~~&x\in\partial \Omega,\\
		u\cdot n=0,~\mathrm{curl}u\times n=0, & x\in\partial B_R,
	\end{cases}
\end{equation}
thanks to \cite{Cai2023,Lili2025}, we have the following conclusions on the solutions to \eqref{2.3}.
\begin{lemma}\label{elliptic_estimate}
Let $u$ be a smooth solution of the Lam\'e system \eqref{2.3},   for any $p\in(1,\infty)$ and $k\geq 0$, there exists a positive constant $C$ depending only on $\lambda$, $\mu$, $A$, $R_0$, $p$, and $k$ such that
\begin{equation}\label{key}
\|\nabla^{k+2}u\|_{L^p(\Omega_R)}\leq C\|f\|_{W^{k,p}(\Omega_R)}+C\|\na u\|_{L^2(\Omega_R)}.
\end{equation}
\end{lemma}
\begin{proof}
For $\eta$ defined in \eqref{rro1}, it holds that $\eta u$ satisfies
%$$\mathcal{L}(\eta u)=\mu\Delta u+(\mu+\lambda)\nabla \mathrm{div}u$$
\be\ba%\label{stokes2}
\begin{cases}\nonumber
\mathcal{L}(\eta u)=\eta f+\tilde{f}, \,\, &x\in \Omega_{0}, \\
		\eta u\cdot n = 0,\ \mathrm{curl}(\eta u)\times n=-Au,\ \
\,\, &x\in\partial \Omega,\\
		\eta u\cdot n = 0,\ \mathrm{curl}(\eta u)\times n=0,\ \
\,\, &x\in\partial B_{2R_0},
\end{cases}
\ea\ee
with
$$\tilde{f}\triangleq\mu u\Delta \eta +2\mu\nabla\eta\cdot\nabla u+(\mu+\lambda)\left(u\nabla^2\eta +\nabla u\cdot\nabla\eta+\nabla\eta\div u\right).$$

Consequently,  one obtains from  \cite[Lemma 2.4]{Cai2023} that
\be\la{vpiq1}\ba
\|\eta u\|_{W^{k+2,p}(\Omega_R)}&=\|\eta u\|_{W^{k+2,p}(\Omega_{0})}\\
&\leq C(\|\eta f\|_{W^{k,p}(\Omega_{0})}+\|\tilde{f}\|_{W^{k,p}(\Omega_{0})}+\|\eta u\|_{L^{p}(\Omega_{0})})\\
&\leq C(\| f\|_{W^{k,p}(\Omega_R)}+\|\na u\|_{L^2(\Omega_R)})+\frac{1}{4}\|u\|_{W^{k+2,p}(\Omega_R)},
\ea\ee
where we have used $u\cdot n=0$ on $\partial\Omega_R$ and Lemma \ref{G-N}.

Next, we can check that $(1-\eta) v$ satisfies
%$$\mathcal{L}(\eta u)=\mu\Delta u+(\mu+\lambda)\nabla \mathrm{div}u$$
\be\ba
\begin{cases}\nonumber
\mathcal{L}\bigr((1-\eta) u\bigr)=(1-\eta) f+\tilde{f}, \,\, &x\in B_R, \\
		(1-\eta) u\cdot n = 0,\ \mathrm{curl}\bigr((1-\eta) u\bigr)\times n=0,\ \
\,\, &x\in\partial B_R,
\end{cases}
\ea\ee
and thus derive from \cite[Lemma 2.3]{Lili2025} that
\be\ba\nonumber
\|(1-\eta) u\|_{W^{k+2,p}(\Omega_R)}
\leq C(\| f\|_ {W^{k,p}(\Omega_R)}+\|\na u\|_{L^2(\Omega_R)})+\frac{1}{4}\|u\|_{W^{k+2,p}(\Omega_R)}.\ea\ee
This together with \eqref{vpiq1} yields \eqref{key} and finished the proof of \ref{elliptic_estimate}.
\end{proof}

%\begin{remark}
%For general external domains $\Omega$, the results of the Lemmas \ref{G-N}-\ref{elliptic_estimate} also hold.%, see \cite{liu2025} for details.
%\end{remark}

Next, in order to obtain the estimate of the boundary terms, we need the following lemma to establish the positive lower bound of the density near the boundary.

\begin{lemma}\label{087}
Assume that $U\subset\mathbb{R}^3$ is a bounded smooth domain satisfying the interior sphere condition, if $\rho\in C(\overline U)$ and 
$\nabla\rho^{-a}\in L^p(U)$ with  $a>0,\ p>2$, 
%then there exists a positive constant $\hat{\rho}$ depending on $U,a,p$, and $\rho$ such that
%$$\rho>\hat{\rho},~~~x\in \overline{U}.$$
then $\rho$ has a positive lower bound  on $\overline U$.
\end{lemma}
\begin{proof}
%First, since $\partial  U$ is smooth, $\Omega$ satisfies interior sphere condition, therefore 
For any $x\in \overline U$, there exists a ball $B_x\subset \overline U$ such that $x\in B_x$. 
Notice that $\nabla\rho^{-a}\in L^p(U)$, 
it is impossible for $\rho$ to vanish on a subset of positive measure. Then there exists $x_0\in B_x$ such that
$$
\rho(x_0)=\frac{1}{|B_x|}\int_{B_x}\rho dx>\tilde C_1>0,
$$
and thus $\rho^{-a}(x_0)\leq \tilde C_2$.

Next, by Sobolev's embedding theorem (\cite[Chapter 5]{Evans2010}), it holds that for any $x, x_0\in \overline B_x$,
$$
\rho^{-a}(x)\leq \rho^{-a}(x_0)+C\|\nabla\rho^{-a}\|_{L^p(B_x)}|x-x_0|^{1-\frac2p}\leq \tilde C_3,
$$
where $\tilde C_3>0)$. This along with the arbitrary of $x\in \overline U$ thus implies  \begin{equation}\label{xiajie}
    \rho(x)>0,~~~x\in \overline U.
\end{equation}
Recalling that  $\rho\in C(\overline U)$, the continuity of $\rho$ together with \eqref{xiajie} yields that   $\rho(x)$ has a positive lower bound over $\overline U$. This completes the proof of Lemma \ref{087}.
\end{proof}

%\begin{remark}
%In order to satisfy the inner sphere condition of $\Omega$, the smoothness of the boundary can be reduced to $C^2$.
%\end{remark}

\section{A priori estimate}
In this section, for $1\leq p\leq\infty$ and positive integer $k$, we denote
\begin{equation*}
	\int f \mathrm{dx}=\int_{\Omega_R} f \mathrm{dx},\ L^p=L^p(\Omega_R),\ W^{k,p} = W^{k,p}(\Omega_R),\ D^{k}=D^{k,2}(\Omega_R),\ H^k=H^k(\Omega_R).
\end{equation*}

We start with the following local well-posedness result which can be proved in the same way as \cite{3,4,5}, where the initial density is strictly away from vacuum.
%First,   the following local existence theory on $\Omega_R$, with $R>2R_0+1$, where the initial density is strictly away from vacuum, can be shown by similar arguments as in \cite{3,4,5}.
\begin{lemma}\label{l25}
Assume that the initial data $(\rho_0,u_0)$ satisfies
	\begin{equation}\label{2.16}
\begin{split}
    &\rho_0\in H^3(\Omega_R),\ u_0\in H^3(\Omega_R),\ \inf_{x\in \Omega_R}\rho_0(x)>0,\\
    &u_0\cdot n=0,\ \mathrm{curl}u_0\times n=-Au_0,\ x\in\partial \Omega,\\
    &u_0\cdot n=0,\ \mathrm{curl}u_0\times n=0,\ x\in\partial B_R.
\end{split}
	\end{equation}
Then there exist a small time $T_R$ and a unique classical solution $(\rho, u)$ to the following IBVP
	\begin{equation}\label{9}
		\left\{ \begin{array}{l}
			\rho_t+\mathrm{div}(\rho u)=0,\\
			\rho^{1-\delta}u_t+\rho^{1-\delta}u\cdot\nabla u-\mathcal{L}u+\frac{a\gamma}{\gamma-\delta}\nabla\rho^{\gamma-\delta}=\delta\nabla\log\rho\cdot \mathcal{S}(u),\\
			u\cdot n=0,\ \mathrm{curl}u\times n=-Au,\ x\in \partial\Omega,\ t>0,\\
			u\cdot n=0,\ \mathrm{curl}u\times n=0,\ x\in \partial B_R,\ t>0,\\
			(\rho, u)(x, 0) = (\rho_0, u_0)(x),\ x\in \Omega_R,
		\end{array} \right.
	\end{equation}
on $\Omega_R\times(0,T_R)$ such that
	\begin{equation}
		\left\{ \begin{array}{l}
			\rho\in C([0,T_R]; H^3),\ \nabla\log\rho\in C([0,T_R]; H^2),\\
		    u\in C([0,T_R]; H^3)\cap L^2(0, T_R; D^4),\\
			u_t \in L^\infty(0, T_R; H_0^1)\cap L^2(0, T_R; H^2),\ \rho^\frac{1-\delta}{2}u_{tt}\in L^2(0, T_R; L^2),\\
			\sqrt tu\in L^\infty(0, T_R; D^4),\ \sqrt tu_t \in L^\infty(0, T_R; D^2),\ \sqrt tu_{tt} \in L^2(0, T_R; H^1),\\
			\sqrt t\rho^\frac{1-\delta}{2} u_{tt} \in L^\infty(0, T_R; L^2),\ tu_t \in L^\infty(0, T_R; D^3),\\
			tu_{tt} \in L^\infty(0, T_R; H^1) \cap L^2(0, T_R; D^2),\ t\rho^\frac{1-\delta}{2} u_{ttt} \in L^2(0, T_R; L^2),\\
			t^{3/2}u_{tt} \in L^\infty(0, T_R; D^2),\ t^{3/2}u_{ttt} \in L^2(0, T_R; H^1),\\
			t^{3/2}\rho^\frac{1-\delta}{2} u_{ttt}\in L^\infty(0, T_R; L^2).
		\end{array} \right.
	\end{equation}
%where we denote $L^2 = L^2(\Omega_R)$, $D^k = D^k(\Omega_R)$, $H^k = H^k(\Omega_R)$ and $W^{k,p} = W^{k,p}(\Omega_R)$ for positive integer $k$.
\end{lemma} 
In the rest of this section, we always assume that $ (\rho, u)$ is a solution  to the  IBVP \eqref{9} in $\Omega_R \times [0, T_R]$, which is obtained by 
Lemma \ref{l25}.  
The main aim of this section is to derive the following key a priori estimate on $\mathcal{E}$ defined by
\begin{equation}
	\begin{split}\label{qaqa}
		\mathcal{E}(t)\triangleq& 1+\|\rho^\frac{1-\delta}{2}u\|_{L^2}
		+\|\nabla u\|_{L^2}
		+\|\rho^\frac{1-\delta}{2}u_t\|_{L^2}
		+\|\rho^{\gamma-\frac{1+\delta}{2}}\|_{W^{1,6}\cap D^1\cap D^2}\\
		&+\|\nabla\rho^\frac{\delta-1}{2}\|_{L^{6}\cap D^{1}}+\|\rho^\frac{\delta-1}{2}\|_{L^{6}(\Omega_0)}.
	\end{split}
\end{equation}
\begin{proposition}\label{p1}
	Assume that $(\rho_0,u_0)$ satisfies \eqref{2.16}. Let $(\rho,u)$ be the solution to the IBVP \eqref{9} in $\Omega_R\times(0,T_R]$ obtained by Lemma \ref{l25}. Then there exist positive constants $T_0$ and $M$ both depending only on $a$, $\delta$, $\gamma$, $\mu$, $\lambda$, $A$, $R_0$, and $C_0$ such that
	\begin{equation}\label{e32}
			\sup_{0\leq t\leq T_0}\big(\mathcal{E}(t)+\|\nabla^2u\|_{L^2}\big)
            +\int_0^{T_0}\bigl(\|\rho^\frac{\delta-1}{2}\mathcal{L}u\|_{L^2}^2
            +\|\nabla u_t\|_{L^2}^2\bigr) ds
			\leq M,
	\end{equation}
    where
    \begin{equation}
    	\begin{split}
    	C_0\triangleq&1+\|\rho_0^\frac{1-\delta}{2}u_0\|_{L^2}+\|u_0\|_{D^1\cap D^2}
    		+\|\rho_0^{\gamma-\frac{1+\delta}{2}}\|_{W^{1,6}\cap D^1\cap D^2}\\
    		&+\|\nabla\rho_0^\frac{\delta-1}{2}\|_{ L^6\cap D^1}+\|\rho_0^\frac{\delta-1}{2}\mathcal{L}u_0\|_{L^2}+\|\rho_0^\frac{\delta-1}{2}\|_{L^{6}(\Omega_0)} 
    	\end{split}
    \end{equation}
is bounded constant due to \eqref{1117}, \eqref{118}, and \eqref{260114}.
\end{proposition}

To prove Proposition \ref{p1} whose proof will be postponed to the end of this section, we will establish some necessary a priori estimates, see Lemmas \ref{L3.1}-\ref{L3.6}.  
We begin with the following standard energy estimate for $(\rho,u)$.
\begin{lemma}\label{L3.1}
Under the conditions of Proposition \ref{p1}, let $(\rho,u)$ be a smooth solution to the IBVP \eqref{9}.
Then there exist a $T_1=T_1(C_0)>0$ and a positive constant $\beta=\beta(\delta,\gamma)>1$ such that for all $t\in (0,T_1]$,
\begin{equation}\label{e34}
\sup_{0\leq s\leq t}\|\rho^\frac{1-\delta}{2}u\|_{L^2}^2+\int_0^t\|\nabla u\|_{L^2}^2ds\leq CC_0^\beta+C\int_0^t\mathcal{E}^\beta ds,
\end{equation}
where (and in what follows) $C$ denotes a generic positive constant depending only on $a$, $\delta$, $\gamma$, $\mu$, $\lambda$, $A$, $R_0$ and $C_0$.
\end{lemma}
\begin{proof}
First, from \eqref{qaqa} it follows that
\begin{equation}\label{rho_sup}
    \|\rho\|_{L^\infty}=\|\rho^{\gamma-\frac{1+\delta}{2}}\|_{L^\infty}^\frac{2}{2\gamma-1-\delta}\leq C\|\rho^{\gamma-\frac{1+\delta}{2}}\|_{W^{1,6}}^\frac{2}{2\gamma-1-\delta}\leq C\mathcal{E}^\beta,
\end{equation}
in which $\beta=\beta(\delta, \gamma)>1$ is allowed to change from line to line.

Next, it follows from \eqref{9}$_2$ that   $u$ satisfies 
\begin{equation}\label{e38}
\mathcal{L}u=\rho^{1-\delta}u_t+\rho^{1-\delta}u\cdot\nabla u+\frac{a\gamma}{\gamma-\delta}\nabla\rho^{\gamma-\delta}-\delta\nabla\log\rho\cdot S(u).
\end{equation}
One thus obtains after using Lemma \ref{elliptic_estimate}, Gagliardo-Nirenberg inequality, \eqref{qaqa},  \eqref{e38}, \eqref{9}$_3$, \eqref{9}$_4$, and \eqref{rho_sup} that
\begin{align}
\|\nabla^2u\|_{L^2}\leq
&C\bigl(\|\rho^{1-\delta}u_t\|_{L^2}+\|\rho^{1-\delta}u\cdot\nabla u\|_{L^2}+\|\nabla\rho^{\gamma-\delta}\|_{L^2}+\|\nabla\log\rho\cdot S( u)\|_{L^2}+\|\nabla u\|_{L^2}\bigr)\nonumber\\
\leq&C\bigl(\|\rho\|_{L^\infty}^{(1-\delta)/2}\|\rho^\frac{1-\delta}{2}u_t\|_{L^2}
+\|\rho\|_{L^\infty}^{1-\delta}\|\nabla u\|_{L^2}^{3/2}\|\nabla u\|_{H^1}^{1/2}
+\|\rho\|_{L^\infty}^{(1-\delta)/2}\|\nabla\rho^{\gamma-\frac{1+\delta}{2}}\|_{L^2}\nonumber\\
&+\|\rho\|_{L^\infty}^{(1-\delta)/2}\|\nabla\rho^{\frac{\delta-1}{2}}\|_{L^{6}}\|\nabla u\|_{L^2}^{1/2}\|\nabla u\|_{H^1}^{1/2}+\|\nabla u\|_{L^2}\bigr)\nonumber\\
\leq&C\mathcal{E}^{\beta}+\frac{1}{2}\|\nabla^2u\|_{L^2}\nonumber,
%\leq &C\mathcal{E}^{\beta},\nonumber
\end{align}
which directly yields that
\begin{align}\label{e39}
\|\nabla^2u\|_{L^2}\leq C\mathcal{E}^{\beta}.
\end{align}

%Additionally, \eqref{e39} together with Lemma \ref{G-N} yields
%\begin{align}\label{u_sup}
%\|u\|_{L^\infty}\leq C\|\nabla u\|_{H^1}\leq C\mathcal{E}^\beta.
%\end{align}

Then, multiplying the mass equation $\eqref{9}_1$ by $(1-\delta)\rho^{-\delta}$ show that
\begin{equation}\label{33}
(\rho^{1-\delta})_t+u\cdot\nabla\rho^{1-\delta}+(1-\delta)\rho^{1-\delta}\mathrm{div}u=0.
\end{equation}
Adding $\eqref{9}_2$ multiplied by  $u$ and  \eqref{33} multiplied by  $\frac{|u|^2}{2}$ together, we obtain after integrating the resulting equality on $\Omega_R$ by parts that 
\begin{equation}\label{3.4}
\begin{split}
&\frac{1}{2}\frac{d}{dt}\|\rho^\frac{1-\delta}{2}u\|_{L^2}^2+\mu\|\mathrm{curl} u\|_{L^2}^2+
(2\mu+\lambda)\|\mathrm{div}u\|_{L^2}^2+\mu\int_{\partial\Omega}u\cdot A\cdot udS \\
&\leq C\int\bigl(\rho^{1-\delta}|\nabla u||u|^2
+|\nabla\rho^{\gamma-\frac{1+\delta}{2}}||\rho^\frac{1-\delta}{2}u|+|\nabla\rho^\frac{\delta-1}{2}||\rho^\frac{1-\delta}{2}u||\nabla u|\bigr)dx\\
&\leq C\Big(\|\rho\|_{L^\infty}^\frac{1-\delta}{2}\|\nabla u\|_{L^2}\|u\|_{L^\infty}\|\rho^\frac{1-\delta}{2} u\|_{L^2}
+\|\nabla\rho^{\gamma-\frac{1+\delta}{2}}\|_{L^2}\|\rho^\frac{1-\delta}{2} u\|_{L^2}\\
&\quad\quad\quad+\|\nabla\rho^\frac{\delta-1}{2}\|_{L^{6}}\|\rho^\frac{1-\delta}{2} u\|_{L^2}\|\nabla u\|_{L^{3}}\Big)\\
&\leq C\mathcal{E}^\beta,
\end{split}
\end{equation}
where one has used Lemma \ref{G-N}, \eqref{qaqa}, \eqref{rho_sup}, and \eqref{e39}.
Using the trace theorem, the boundary term in \eqref{3.4} can be governed as
\begin{align}\label{trace1}
\int_{\partial\Omega}u\cdot A\cdot udS\leq C\|u\|^2_{H^1(\Omega_0)}  \leq C(\Omega_0)\|\nabla u\|_{L^2}^2\leq C\mathcal{E}^\beta.
\end{align}

%Combining with the fact
%\begin{equation}\label{3.4aqa}
%\begin{split}
%\int_0^t\|\nabla u\|_{L^2}^2ds\leq \int_0^t\mathcal{E}^2ds,
%\end{split}
%\end{equation}
Thus, one derives \eqref{e34} after  integrating the resulting \eqref{3.4} over $(0,t)$ and using \eqref{trace1}. The proof of Lemma \ref{L3.1} is completed.
\end{proof}

The following lemma shows the weighted estimates for $u_t$ with positive power    of $\rho$ and  the  $L^2$-bounds for  $\nabla u_t$.

\begin{lemma}\label{L3.3}
	Let $(\rho,u)$ and $T_1$ be as in Lemma \ref{L3.1}. Then for all $t\in (0,T_1]$,
	\begin{equation}\label{e313}
		\sup_{0\leq s\leq t}\|\rho^\frac{1-\delta}{2}u_t\|_{L^2}^2+\int_0^t\|\nabla u_t\|_{L^2}^2ds
		\leq CC_0^\beta+C\int_0^t\mathcal{E}^\beta ds.
	\end{equation}
\end{lemma}
\begin{proof}
First, it follows from $\eqref{9}_1$ that $\rho^{\gamma-\delta}$ and $\nabla\log\rho$ satisfy
\begin{equation}
\label{qa98}
(\rho^{\gamma-\delta})_t=-u\cdot\nabla\rho^{\gamma-\delta}-(\gamma-\delta)\rho^{\gamma-\delta}\mathrm{div}u, 
\end{equation}
and
\begin{equation}
\label{qa981} 
(\nabla\log\rho)_t=-\nabla u\cdot \nabla\log\rho-u\cdot \nabla^2\log\rho-\nabla\mathrm{div}u,
\end{equation}
which combining with Lemma \ref{G-N}, \eqref{qaqa}, \eqref{rho_sup}, and \eqref{e39} that
\begin{gather}
%\|(\rho^\frac{1-\delta}{2})_t\|_{L^3}\leq\|\rho\|_{L^\infty}^{1-\delta}\|u\|_{L^6}\|\nabla\rho^{\frac{\delta-1}{2}}\|_{L^6}+\|\rho\|_{L^\infty}^{(1-\delta)/2}\|\nabla u\|_{L^3}\leq C\mathcal{E}^\beta,\label{qaz1}\\
\|(\rho^{\gamma-\delta})_t\|_{L^2}\leq\|\rho\|_{L^\infty}^{(1-\delta)/2}\|u\|_{L^\infty}\|\nabla\rho^{\gamma-\frac{1+\delta}{2}}\|_{L^2}+\|\rho\|_{L^\infty}^{\gamma-\delta}\|\nabla u\|_{L^2}\leq C\mathcal{E}^\beta,\label{qaz2}
%\|(\nabla\log\rho)_t\|_{L^2}\leq C\|\rho\|_{L^\infty}^{(1-\delta)/2}\|u\|_{L^6}\|\nabla\rho^\frac{\delta-1}{2}\|_{L^{6}}+\|\nabla u\|_{L^3}\leq C\mathcal{E}^\beta,\label{qaz3}
\end{gather}
%\begin{equation}\label{3.38}
%\begin{split}
%&\|\rho^\frac{\delta-1}{2}\nabla(\rho^{\gamma-\delta})_t\|_{L^2}\\
%&\leq  C\bigl(\|\rho\|_{L^\infty}^{(2\gamma-1-\delta)/2}\|\nabla^2u\|_{L^2}
%+\|\nabla\rho^{\gamma-\frac{\delta+1}{2}}\|_{L^2}\|\nabla u\|_{L^2}
%+\|\nabla^2\rho^{\gamma-\frac{\delta+1}{2}}\|_{L^2}\|u\|_{L^\infty}\\ &\quad+\|\nabla\rho^\frac{\delta-1}{2}\|_{L^{6}}\|\nabla\rho^{\gamma-\frac{\delta+1}{2}}\|_{L^6}\|\rho^\frac{1-\delta}{2}u\|_{L^{6}}\bigr)\\
%&\leq C\mathcal{E}^\beta,
%\end{split}
%\end{equation}
\begin{equation}
\begin{split}\label{qaz4}
\|\nabla(\log\rho)_t\|_{L^2}
&\leq C\bigl(\|\rho\|_{L^\infty}^{(1-\delta)/2}\|\nabla u\|_{L^{3}}\|\nabla\rho^\frac{\delta-1}{2}\|_{L^{6}}+\|\rho\|_{L^\infty}^{1-\delta}\|u\|_{L^{6}}\|\nabla\rho^\frac{\delta-1}{2}\|_{L^{6}}^2\\
&\quad+\|\rho\|_{L^\infty}^{(1-\delta)/2}\|u\|_{L^\infty}\|\nabla^2\rho^\frac{\delta-1}{2}\|_{L^{2}}
+\|\nabla\mathrm{div}u\|_{L^2}\bigr)\\
&\leq C\mathcal{E}^\beta.
\end{split}
\end{equation}

Next, operating $\partial_t$ to $\eqref{9}_2$ yields that
\begin{equation}\label{32}
\begin{split}
(\rho^{1-\delta})_tu_t+\rho^{1-\delta}u_{tt}
+(\rho^{1-\delta})_tu\cdot\nabla u
+\rho^{1-\delta}u_t\cdot\nabla u
+\rho^{1-\delta}u\cdot\nabla u_t\\
-\mathcal{L}u_t+\frac{a\gamma}{\gamma-\delta}\nabla(\rho^{\gamma-\delta})_t
=\delta(\nabla\log\rho)_t\cdot S(u)+\delta\nabla\log\rho\cdot S(u)_t.
\end{split}
\end{equation}
%We multiply the mass equation $\eqref{9}_1$ by $\rho^{-1}$ to get
%\begin{equation}\label{3.13}
%	(\log\rho)_t=-\nabla\log\rho\cdot u-\mathrm{div}u.
%\end{equation}
Multiplying the above equality by $u_t$, we obtain after using integration by parts and $\eqref{33}$ that
\begin{equation}\label{34}
\begin{split}
&\frac{1}{2}\frac{d}{dt}\|\rho^\frac{1-\delta}{2}u_t\|_{L^2}^2+\mu\|\mathrm{curl} u_t\|_{L^2}^2
+(2\mu+\lambda)\|\mathrm{div}u_t\|_{L^2}^2+\mu\int_{\partial\Omega}u_t\cdot A\cdot u_tdS \\
&\leq C\int\Big(\rho^{1-\delta}|\nabla u||u_t|^2+\rho^{1-\delta}|u||\nabla u_t||u_t|+|u|^2|\nabla\rho^{1-\delta}||\nabla u||u_t|+\rho^{1-\delta}|u||\nabla u|^2|u_t|
\\
&\quad+|(\rho^{\gamma-\delta})_t||\div u_t|
+|(\nabla\log\rho)_t||\nabla u||u_t|+|\nabla\log\rho||\nabla u_t||u_t|\Big)dx\\
&\triangleq\sum_{i=1}^7 I_i.
\end{split}
\end{equation}
Thanks to Lemma \ref{G-N}, \eqref{qaqa}, \eqref{rho_sup}, \eqref{e39}, and \eqref{qaz2}-\eqref{qaz4}, we can estimate each $I_i$ as follows:
$$
\begin{aligned}
&I_1\leq C\|\rho\|_{L^\infty}^{(1-\delta)/2}\|\nabla u\|_{L^3}\|\rho^\frac{1-\delta}{2}u_t\|_{L^2}\|u_t\|_{L^6}
\leq C\mathcal{E}^\beta+\epsilon\|\nabla u_t\|_{L^2}^2,\\
&I_2\leq C\|\rho\|_{L^\infty}^{(1-\delta)/2}\|u\|_{L^\infty}\|\nabla u_t\|_{L^2}\|\rho^\frac{1-\delta}{2}u_t\|_{L^2}
\leq C\mathcal{E}^\beta+\epsilon\|\nabla u_t\|_{L^2}^2,\\
&I_3\leq C\|\rho\|_{L^\infty}^{3(1-\delta)/2}\|u\|_{L^6}\|u\|_{L^\infty}\|\nabla\rho^\frac{\delta-1}{2}\|_{L^{6}}\|\nabla u\|_{L^2}\|u_t\|_{L^6}
\leq C\mathcal{E}^\beta+\epsilon\|\nabla u_t\|_{L^2}^2,\\
&I_4\leq C\|\rho\|_{L^\infty}^{1-\delta}\|u\|_{L^6}\|\nabla u\|_{L^2}\|\nabla u\|_{L^6}\|u_t\|_{L^6}
\leq C\mathcal{E}^\beta+\epsilon\|\nabla u_t\|_{L^2}^2,\\
&I_5\leq C\|(\rho^{\gamma-\delta})_t\|_{L^2}\|\nabla u_t\|_{L^2}
\leq C\mathcal{E}^\beta+\epsilon\|\nabla u_t\|_{L^2}^2,\\
&I_6\leq C\|(\nabla\log\rho)_t\|_{L^2}\|\nabla u\|_{L^3}\|u_t\|_{L^6}
\leq C\mathcal{E}^\beta+\epsilon\|\nabla u_t\|_{L^2}^2,\\
&I_7 \leq C\|\rho\|_{L^\infty}^{(1-\delta)/4}\|\nabla\rho^\frac{\delta-1}{2}\|_{L^{6}}\|\nabla u_t\|_{L^2}\| u_t\|_{L^6}^{1/2}\|\rho^\frac{1-\delta}{2}u_t\|_{L^2}^{1/2}\leq C\mathcal{E}^\beta+\epsilon\|\nabla u_t\|_{L^2}^2.
\end{aligned}
$$
%where we use the continuity equation and \eqref{3.13} to estimate
%\begin{gather}
%    \|(\rho^{\gamma-\delta})_t\|_{L^2}
%    \leq\|\rho\|_{L^\infty}^{\gamma-\delta}\|\nabla u\|_{L^2}
%    +\|\rho\|_{L^\infty}^{(1-\delta)/2}\|u\|_{L^\infty}\|\nabla\rho^{\gamma-\frac{1+\delta}{2}}\|_{L^2}
%    \leq C\mathcal{E}^{\beta},\\
%    \|(\log\rho)_t\|_{L^3}\leq C\|\rho\|_{L^\infty}^\frac{(1-\delta)}{2}\|\nabla\rho^\frac{\delta-1}{2}\|_{L^{6}}\|\nabla u\|_{H^1}+\|\nabla u\|_{H^1}
%    \leq C\mathcal{E}^{\beta}.
%\end{gather}

Using the trace theorem, the boundary term in \eqref{34} can be governed as
\begin{equation}\label{trace3}
\begin{split}
\int_{\partial\Omega}u_t\cdot A\cdot u_tdS
&\leq C\left\||u_t|^2\right\|_{W^{1,1}(\Omega_0)}  \\
&\leq C\|\rho^{\delta-1}\|_{L^\infty(\Omega_0)}\|\rho^\frac{1-\delta}2u_t\|_{L^2(\Omega_0)}^2+C\|\rho^\frac{\delta-1}2\|_{L^\infty(\Omega_0)}\|\rho^\frac{1-\delta}2u_t\|_{L^2}\|\nabla u_t\|_{L^2}\\
&\leq C\mathcal{E}^\beta\|\rho^\frac{1-\delta}{2}u_t\|_{L^2}^2+\epsilon\|\nabla u_t\|_{L^2}^2,
\end{split}
\end{equation}
where we have used
\begin{equation}\label{rho_inf}
    \|\rho^\frac{\delta-1}{2}\|_{L^\infty(\Omega_0)}\leq C\|\rho^\frac{\delta-1}{2}\|_{W^{1,6}(\Omega_0)}\leq C\mathcal{E}^\beta.
\end{equation}
And, it follows from  Lemma \ref{l22}, \eqref{qaqa} and \eqref{rho_inf} that
\begin{equation}\label{3.4aqa0}
\begin{split}
\int_\tau^t\|\nabla u_t\|_{L^2}^2ds
&\leq C\int_\tau^t\big(\|\curl u_t\|_{L^2}^2+\|\div u_t\|_{L^2}^2+\|u_t\|_{L^2(\Omega_0)}^2\big)ds\\
&\leq C\int_\tau^t\big(\|\curl u_t\|_{L^2}^2+\|\div u_t\|_{L^2}^2+\|\rho^{\delta-1}\|_{L^\infty(\Omega_0)}\|\rho^\frac{1-\delta}2u_t\|_{L^2(\Omega_0)}^2\big)ds\\
&\leq C\int_\tau^t\big(\|\curl u_t\|_{L^2}^2+\|\div u_t\|_{L^2}^2+\mathcal{E}^\beta(s)\big)ds.
\end{split}
\end{equation}

Substituting $I_1$--$I_7$ into \eqref{34} and integrating the resulting inequality over $(\tau,t)$ with $\tau\in(0,t)$, by virtue of \eqref{trace3} and \eqref{3.4aqa0},  we obtain after choosing $\epsilon$  small enough  that
\begin{equation}
\sup_{\tau\leq s\leq t}\|\rho^\frac{1-\delta}{2}u_t(t)\|_{L^2}^2+\int_\tau^t\|\nabla u_t(s)\|_{L^2}^2ds\leq\|\rho^\frac{1-\delta}{2}u_t(\tau)\|_{L^2}^2+C\int_\tau^t\mathcal{E}^\beta(s)ds. \label{35}
\end{equation}

Then, multiplying the reformulated momentum equation $\eqref{9}_2$ by $u_t$, we can easily deduce that
\begin{equation}
\begin{split}\nonumber
\|\rho^\frac{1-\delta}{2}u_t\|_{L^2}
\leq& C\Big(\|\rho
\|_{L^\infty}^{(1-\delta)/2}\|u\|_{L^\infty}\|\nabla u\|_{L^2}+\|\rho^\frac{\delta-1}{2}\mathcal{L}u\|_{L^2}+\|\nabla\rho^{\gamma-\frac{1+\delta}{2}}\|_{L^2}+\|\nabla\rho^\frac{\delta-1}{2}\|_{L^{6}}\|\nabla u\|_{L^{3}}\Big),
\end{split}
\end{equation}
which along with the compatibility condition \eqref{118} implies
\begin{equation}\label{317}
\begin{split}
\limsup\limits_{\tau \to 0}\|\rho^\frac{1-\delta}{2}u_t(\tau)\|_{L^2}
\leq& C\bigl(\|\rho_0\|_{L^\infty}^{(1-\delta)/2}\|u_0\|_{L^\infty}\|\nabla u_0\|_{L^2}
+\|\rho_0^\frac{\delta-1}{2}\mathcal{L}u_0\|_{L^2}\\
&+\|\nabla\rho_0^{\gamma-\frac{1+\delta}{2}}\|_{L^2}
+\|\nabla\rho_0^\frac{\delta-1}{2}\|_{L^{6}}\|\nabla u_0\|_{L^3}\bigr)\\
\leq& C\bigl(\|\rho_0^{\gamma-\frac{1+\delta}{2}}\|_{W^{1,6}}^\frac{1-\delta}{2\gamma-1-\delta}\|\nabla u_0\|_{H^1}\|\nabla u_0\|_{L^2}
+\|\rho_0^\frac{\delta-1}{2}\mathcal{L}u_0\|_{L^2}\\
&+\|\nabla\rho_0^{\gamma-\frac{1+\delta}{2}}\|_{L^2}
+\|\nabla\rho_0^\frac{\delta-1}{2}\|_{L^{6}}\|\nabla u_0\|_{H^1}\bigr)
\\
\leq& C_0^\beta.
\end{split}
\end{equation}

Finally, letting $\tau \to 0$ in \eqref{35}, we can obtain \eqref{e313} after using \eqref{317}.
%\begin{equation}\label{318}
%\|\rho^\frac{1-\delta}{2}u_t(t)\|_{L^2}^2+\int_0^t\|\nabla u_t(s)\|_{L^2}^2ds
%\leq C_0^\beta+C\int_0^t\mathcal{E}^\beta(s)ds.
%\end{equation}
The proof of Lemma \ref{L3.3} is completed.
\end{proof}

Next, we need the following weighted estimates for $\mathcal{L}u$ with negative power  of $\rho$, as well as the  $L^2$-bounds for  both $\curl u$ and $\div u$, which are used for deriving the estimates on $\|\nabla u\|_{L^2}^2$.

\begin{lemma}\label{L3.2}
	Let $(\rho,u)$ and $T_1$ be as in Lemma \ref{L3.1}. Then for all $t\in (0,T_1]$,
	\begin{equation}\label{e310}
		\sup_{0\leq s\leq t}\big(\|\curl u\|_{L^2}^2+\|\div u\|_{L^2}^2\big)+\int_0^t\|\rho^\frac{\delta-1}{2}\mathcal{L}u\|_{L^2}^2ds
		\leq CC_0^\beta+C\int_0^t\mathcal{E}^\beta ds.
	\end{equation}
\end{lemma}
\begin{proof}
Multiplying %the momentum equation
$\eqref{9}_2$ by $-2\rho^{\delta-1}\mathcal{L}u$ and integrating the resulting equality by parts leads to
\begin{equation}
\begin{split}\label{trz}
&\frac{d}{dt}\left(\mu\|\mathrm{curl}u\|_{L^2}^2
+(2\mu+\lambda)\|\mathrm{div}u\|_{L^2}^2\right)
+2\|\rho^\frac{\delta-1}{2}\mathcal{L}u\|_{L^2}^2+2\mu\int_{\partial\Omega} u\cdot A\cdot u_tdS\\
&\leq C\int\bigl(\rho^\frac{1-\delta}{2}|u||\nabla u||\rho^\frac{\delta-1}{2}\mathcal{L}u|+|\nabla\rho^{\gamma-\frac{1+\delta}{2}}||\rho^\frac{\delta-1}{2}\mathcal{L}u|
+|\nabla\rho^\frac{\delta-1}{2}||\nabla u||\rho^\frac{\delta-1}{2}\mathcal{L}u|\bigr)dx\\
&\leq C\Big(\|\rho\|_{L^\infty}^{(1-\delta)/2}\|u\|_{L^\infty}\|\nabla u\|_{L^2}+\|\nabla\rho^{\gamma-\frac{1+\delta}{2}}\|_{L^2}
+\|\nabla\rho^\frac{\delta-1}{2}\|_{L^{6}}\|\nabla u\|_{L^3}\Big)\|\rho^\frac{\delta-1}{2}\mathcal{L}u\|_{L^2}\\
&\leq \|\rho^\frac{\delta-1}{2}\mathcal{L}u\|_{L^2}^2+C\mathcal{E}^\beta,
\end{split}
\end{equation}
where we used Lemma \ref{G-N}, \eqref{qaqa}, \eqref{rho_sup}, and \eqref{e39}.
Similar to \eqref{trace1}, the boundary term on the right hand of  \eqref{trz} can be governed as
\begin{align}\label{trace2}
\int_{\partial\Omega} u\cdot A\cdot u_tdS\leq C\||u||u_t|\|_{W^{1,1}(\Omega_0)}  &\leq C\mathcal{E}^\beta+C\|\nabla u_t\|_{L^2}^2.
\end{align}

Integrating   \eqref{trz} over $(0,t)$, we obtain \eqref{e310} after using \eqref{trace2} and \eqref{e313}. The proof of Lemma \ref{L3.2} is finished.
\end{proof}

%Since $\gamma-\frac{1+\delta}{2}>0$, 
The following lemma is concerning the  estimate of $\rho$ with positive power $\gamma-\frac{1+\delta}{2}>0$.
\begin{lemma}\label{L3.5}
	Let $(\rho,u)$ and $T_1$ be as in Lemma \ref{L3.1}. Then for all $t\in (0,T_1]$,
	\begin{equation}\label{e336}
		\sup_{0\leq s\leq t} \|\rho^{\gamma-\frac{1+\delta}{2}}\|_{W^{1,6}\cap D^1\cap D^2}\leq CC_0^\beta+C\int_0^t\mathcal{E}^\beta ds.
	\end{equation}
\end{lemma}
\begin{proof}
First, it follows from Lemma \ref{G-N}, \ref{elliptic_estimate}, \eqref{9}$_2$, \eqref{9}$_4$, \eqref{e39}, and \eqref{rho_sup} that
\begin{equation}\label{g3u1}\notag
\begin{split}
	\|\nabla^3u\|_{L^2}&\leq C\|\mathcal{L}u\|_{H^1}+C\|\nabla u\|_{L^2}\\
	&\leq C\Big(\|\nabla\rho^\frac{1-\delta}{2}\|_{L^{6}}\|\rho^\frac{1-\delta}{2}u_t\|_{L^{3}}
	+\|\rho^{1-\delta}\|_{L^\infty}\|\nabla u_t\|_{L^2}+\|\nabla\rho^{1-\delta}\|_{L^{6}}\|u\|_{L^6}\|\nabla u\|_{L^{6}}\\
	&\quad+\|\rho^{1-\delta}\|_{L^\infty}\|\nabla u\|_{L^3}\|\nabla u\|_{L^6}
	+\|\rho^{1-\delta}\|_{L^\infty}\|u\|_{L^\infty}\|\nabla^2u\|_{L^2}
	+\|\nabla^2\rho^{\gamma-\delta}\|_{L^2}\\
	&\quad+\|\rho^\frac{1-\delta}{2}\|_{L^\infty}\|\nabla^2\rho^\frac{\delta-1}{2}\|_{L^{2}}\|\nabla u\|_{L^{\infty}}+\|\nabla\rho^\frac{1-\delta}{2}\|_{L^{6}}\|\nabla\rho^\frac{\delta-1}{2}\|_{L^{6}}\|\nabla u\|_{L^{6}}\\
	&\quad+\|\rho^\frac{1-\delta}{2}\|_{L^\infty}\|\nabla\rho^\frac{\delta-1}{2}\|_{L^{6}}\|\nabla^2u\|_{L^{3}}\Big)+C\mathcal{E}^\beta\\
	&\leq C\Big(\|\rho\|_{L^\infty}^{5(1-\delta)/4}
    \|\nabla\rho^{\frac{\delta-1}{2}}\|_{L^{6}}
    \|\rho^\frac{1-\delta}{2}u_t\|_{L^2}^{1/2}\|\nabla u_t\|_{L^2}^{1/2}+\|\rho\|_{L^\infty}^{1-\delta}\|\nabla u_t\|_{L^2}\\
	&\quad
	+\|\rho\|_{L^\infty}^{3(1-\delta)/2}\|\nabla\rho^{\frac{\delta-1}{2}}\|_{L^{6}}
    \|\nabla u\|_{H^1}^2+\|\rho\|_{L^\infty}^{1-\delta}
    \|\nabla u\|_{H^1}^2
	+\|\rho\|_{L^\infty}^{(1-\delta)/2}\|\nabla^2\rho^{\gamma-\frac{1+\delta}{2}}\|_{L^2}\\
	&\quad +\|\rho\|_{L^\infty}^{1-\delta}\|\nabla\rho^{\gamma-\frac{1+\delta}{2}}\|_{L^3}\|\nabla\rho^{\frac{\delta-1}{2}}\|_{L^{6}}
     +\|\rho\|_{L^\infty}^{(1-\delta)/2}
    \|\nabla\rho^{\frac{\delta-1}{2}}\|_{L^6\cap D^1}\|\nabla^2u\|_{L^2}^{1/2}\|\nabla^2u\|_{H^1}^{1/2}\\
	&\quad+\|\rho\|_{L^\infty}^{1-\delta}\|\nabla\rho^{\frac{\delta-1}{2}}\|_{L^{6}}^2\|\nabla u\|_{H^1}+\mathcal{E}^\beta\Big)
\\
	&\leq C\mathcal{E}^\beta+C\mathcal{E}^\beta\|\nabla u_t\|_{L^2}+\frac{1}{2}\|\nabla^3u\|_{L^2},
\end{split}
\end{equation}
which directly yields that
\begin{equation}\label{g3u}
\begin{split}
	\|\nabla^3u\|_{L^2} \leq C\mathcal{E}^\beta+C\mathcal{E}^\beta\|\nabla u_t\|_{L^2}.
\end{split}
\end{equation}

Next, multiplying the mass equation $\eqref{9}_1$ by $(\gamma-\frac{1+\delta}{2})\rho^{\gamma-\frac{1+\delta}{2}-1}$, we find that $\varphi\triangleq\rho^{\gamma-\frac{1+\delta}{2}}$ satisfies
\begin{equation}\label{310}
\varphi_t+u\cdot\nabla\varphi+(\gamma-\frac{1+\delta}{2})\varphi\mathrm{div}u=0.
\end{equation}
 Integrating \eqref{310} multiplied by $\varphi^5$   on $\Omega_R$ leads to
\begin{equation}\label{e337}
	\frac{d}{dt}\|\varphi\|_{L^6}^6
	\leq C\|\mathrm{div} u\|_{L^\infty}\|\varphi\|_{L^6}^6.
\end{equation}
Operating  $\nabla$ to \eqref{310}, one obtains after multiplying the resulting equality by $|\nabla\varphi|^{r-2}\nabla\varphi$ with $r=2,6$  and integrating by parts on $\Omega_R$ that
\begin{equation}\label{e339}
\begin{split}
	\frac{d}{dt}\|\nabla\varphi\|_{L^r}^r
	\leq& C\int\left(|\nabla\varphi|^r|\nabla u|+|\varphi||\nabla\varphi|^{r-1}|\nabla\mathrm{div}u|\right)dx\\
	\leq&C\bigl(\|\nabla u\|_{L^\infty}\|\nabla\varphi\|_{L^r}^r+\|\varphi\|_{L^\infty}\|\nabla\varphi\|_{L^r}^{r-1}\|\nabla\mathrm{div}u\|_{L^r}\bigr).
\end{split}
\end{equation}
Furthermore, operating  $\nabla^2$ to \eqref{310}, multiplying it by $\nabla^2\varphi$ and integrating by parts on $\Omega_R$, one has
\begin{equation}\label{e340}
\begin{split}
\frac{d}{dt}\|\nabla^2\varphi\|_{L^2}^2
\leq& C\int\bigl(|\nabla^2\varphi|^2|\nabla u|+|\nabla\varphi||\nabla^2\varphi||\nabla^2u|+|\varphi||\nabla^2\varphi||\nabla^2\mathrm{div}u|\bigr)dx\\
\leq&C\bigl(\|\nabla u\|_{L^\infty}\|\nabla^2\varphi\|_{L^2}^2+\|\nabla\varphi\|_{L^6}\|\nabla^2\varphi\|_{L^2}\|\nabla^2u\|_{L^3}\\
&\quad\quad+\|\varphi\|_{L^\infty}\|\nabla^2\varphi\|_{L^2}\|\nabla^3u\|_{L^2}\bigr).
\end{split}
\end{equation}

Finally, we deduce from \eqref{e337}-\eqref{e340}, \eqref{g3u}, and Gagliardo-Nirenberg inequality that
\begin{equation}\label{260129}
\frac{d}{dt}\|\varphi\|_{W^{1,6}\cap D^1\cap D^2}\leq C\|\nabla u\|_{H^2}\|\varphi\|_{W^{1,6}\cap D^1\cap D^2}
\leq C\mathcal{E}^\beta+C\|\nabla u_t\|_{L^2}^2.
\end{equation}
Thus, \eqref{e336} is derived from integrating \eqref{260129} over $(0,t)$ and using \eqref{e313}. 
%\begin{equation}%\label{3.36}
%\begin{split}\nonumber
%\sup_{0\leq s\leq t}\|\varphi\|_{W^{1,6}\cap D^1\cap D^2}
%&\leq C\|\varphi_0\|_{W^{1,6}\cap D^1\cap D^2}\exp\bigg(C\int_0^t\|\nabla u\|_{H^2}ds\bigg)\\
%&\leq C\exp\bigg(C\int_0^t\bigl(\mathcal{E}^\beta+\|\nabla u_t\|_{L^2}^2\bigr)ds\bigg)\\
%&\leq C\exp\bigg(C\int_0^t\mathcal{E}^\beta ds\bigg).
%\end{split}
%\end{equation}
The proof of Lemma \ref{L3.5} is finished.
\end{proof}

Now, we are in a position to derive the crucial estimates on $\rho^\frac{\delta-1}{2}$. 
\begin{lemma}\label{L3.6}
	Let $(\rho,u)$ and $T_1$ be as in Lemma \ref{L3.1}. Then for all $t\in (0,T_1]$,
	\begin{equation}\label{e343}
		\sup_{0\leq s\leq t}\bigr(\|\rho^\frac{\delta-1}{2}\|_{L^6(\Omega_0)}+\|\nabla\rho^\frac{\delta-1}{2}\|_{L^{6}\cap D^{1}}\bigr)
		\leq CC_0^\beta+C\int_0^t\mathcal{E}^\beta ds.
	\end{equation}
\end{lemma}
\begin{proof}
First, one can deduce from  the boundary condition $\mathrm{curl}u\times n=0$ on $\partial B_R$ that
\begin{align}\label{curlu}
\mathrm{curlcurl}u\cdot n=0\,\ \text{on}\ \ \partial B_R.
\end{align}
Indeed,  for all $\vartheta\in C^\infty(\overline{\Omega_R})$ with $\vartheta=\nabla \vartheta=0$ on $\partial\Omega$, we have
\begin{equation}
\begin{split}\nonumber
    \int_{\partial B_R} \mathrm{curlcurl} u\cdot n \vartheta dS&=\int_{\partial \Omega_R} \mathrm{curlcurl} u\cdot n \vartheta dS
    =\int_{\Omega_R} \nabla\vartheta\cdot \mathrm{curlcurl} u dx\\
    &=-\int_{\Omega_R}\mathrm{div}\bigl( \nabla\vartheta\times \mathrm{curl} u\bigr) dx=-\int_{\partial \Omega_R}\bigl( \nabla\vartheta\times \mathrm{curl} u\bigr)\cdot n dS\\
    &=-\int_{\partial \Omega_R}\bigl( \mathrm{curl} u\times  n \bigr)\cdot \nabla\vartheta dS=\int_{\partial B_R} \bigl( \mathrm{curl} u\times  n \bigr)\cdot \nabla\vartheta dS=0,
\end{split}
\end{equation}
which thus yields directly  \eqref{curlu}.  
%\begin{equation}\label{cvt}
%\mathrm{curlcurl}u\cdot n=0,\,\,\ \text{on}\,\,\ \partial \Omega_R.
%\end{equation}
% This fact yields
%\begin{equation}\label{cvt}
%\mathrm{curlcurl}u\cdot n=0,\,\,\ \text{on}\,\,\ \partial \Omega_R.
%\end{equation}

Next, we obtain by direct calculations
\begin{equation}\label{3.32}
\begin{split}
&\|\rho^\frac{\delta-1}{2}\mathcal{L}u\|_{L^2}^2\\
&=(\lambda+2\mu)^2\|\rho^\frac{\delta-1}{2}\nabla \mathrm{div}u\|_{L^2}^2+\mu^2\|\rho^\frac{\delta-1}{2}\mathrm{curlcurl}u\|_{L^2}^2\\
&\quad-2\mu(\lambda+2\mu)\int_{\Omega_R}\rho^{\delta-1}\nabla\mathrm{div}u\cdot\mathrm{curlcurl}udx\\
&=(\lambda+2\mu)^2\|\rho^\frac{\delta-1}{2}\nabla \mathrm{div}u\|_{L^2}^2+\mu^2\|\rho^\frac{\delta-1}{2}\mathrm{curlcurl}u\|_{L^2}^2\\
&\quad-2\mu(\lambda+2\mu)\left(\int_{\partial\Omega_R}\rho^{\delta-1}\mathrm{div}u\mathrm{curlcurl}u\cdot ndS-\int_{\Omega_R}\nabla\rho^{\delta-1}\cdot\mathrm{curlcurl}u\mathrm{div}udx\right)\\
&\geq(\lambda+2\mu)^2\|\rho^\frac{\delta-1}{2}\nabla\mathrm{div}u\|_{L^2}^2+\mu^2/2\|\rho^\frac{\delta-1}{2}\mathrm{curlcurl}u\|_{L^2}^2-C\mathcal{E}^\beta\||\nabla u||\nabla^2 u|\|_{W^{1,1}(\Omega_0)}
\\
&\quad-C\|\nabla\rho^\frac{\delta-1}{2}\|_{L^{6}}^2\|\mathrm{div}u\|_{L^{3}}^2,
%\geq&(\lambda+2\mu)^2\|\nabla\big(\rho^\frac{\delta-1}{2}\mathrm{div}u\big)\|_{L^2}^2
%-C\|\nabla\rho^\frac{\delta-1}{2}\|_{L^{6}}^2\|\mathrm{div}u\|_{L^{3}}^2,
\end{split}
\end{equation}
where we have used \eqref{e39}, \eqref{rho_inf}, \eqref{g3u}, and \eqref{curlu}.
Hence $w\triangleq\rho^\frac{\delta-1}{2}\nabla\mathrm{div}u$ and $\rho^\frac{\delta-1}{2}\mathrm{curlcurl}u$ satisfies the following estimates
%Notice that $w\triangleq\rho^\frac{\delta-1}{2}\nabla\mathrm{div}u$ satisfies \eqref{260122}, one has the following estimates
\begin{equation}\label{114}
\|w\|_{L^2}+\|\rho^\frac{\delta-1}{2}\mathrm{curlcurl}u\|_{L^2}
%\leq C\bigl(\|\rho^\frac{\delta-1}{2}\mathcal{L}u\|_{L^2}+\|\nabla\rho^\frac{\delta-1}{2}\|_{L^{6}}\|\nabla u\|_{H^1}\bigr)
\leq C\|\rho^\frac{\delta-1}{2}\mathcal{L}u\|_{L^2}+ C\mathcal{E}^\beta+C\|\nabla u_t\|_{L^2}.
\end{equation}
Multiplying $\eqref{9}_2$ by $\rho^\frac{\delta-1}{2}$ gives that 
\begin{equation}
	\begin{split}\label{zsa1}
\rho^\frac{\delta-1}{2}\mathcal{L} u=&\rho^\frac{1-\delta}{2}u_t+\rho^\frac{1-\delta}{2}u\cdot\nabla u+\frac{2a\gamma}{2\gamma-1-\delta}\nabla\rho^{\gamma-\frac{1+\delta}{2}}+\frac{2}{1-\delta}\nabla\rho^\frac{\delta-1}{2}\cdot\mathcal{S}(u),
    \end{split}
\end{equation}
and thus 
\begin{equation}
\begin{split}\label{zsa2}
\mu\rho^\frac{\delta-1}{2}\curl\curl u=&-\rho^\frac{1-\delta}{2}u_t-\rho^\frac{1-\delta}{2}u\cdot\nabla u-\frac{2a\gamma}{2\gamma-1-\delta}\nabla\rho^{\gamma-\frac{1+\delta}{2}}-\frac{2}{1-\delta}\nabla\rho^\frac{\delta-1}{2}\cdot\mathcal{S}(u)\\
&+(2\mu+\lambda)\nabla\big(\rho^\frac{\delta-1}{2}\div u\big)-(2\mu+\lambda)\nabla\rho^\frac{\delta-1}{2}\div u.
\end{split}
\end{equation}
It follows from Lemma \ref{G-N}, \eqref{zsa1}-\eqref{zsa2}, \eqref{e39}, \eqref{rho_sup}, \eqref{g3u}, and \eqref{114} that
\begin{equation}%\label{e349}
\begin{split}\nonumber
&\|\nabla(\rho^\frac{\delta-1}{2}\mathcal{L} u)\|_{L^2}+\|\curl\left(\rho^\frac{\delta-1}{2}\curl\curl u\right)\|_{L^2}\\
&\leq C\Big(\|\rho^\frac{1-\delta}{2}\|_{L^\infty}\|\nabla\rho^\frac{\delta-1}{2}\|_{L^{6}}\|\rho^\frac{1-\delta}{2}u_t\|_{L^3}
+\|\rho^\frac{1-\delta}{2}\|_{L^\infty}\|\nabla u_t\|_{L^{2}}\\
&\quad+\|\rho^{1-\delta}\|_{L^\infty}\|\nabla\rho^\frac{\delta-1}{2}\|_{L^{6}}\|u\|_{L^6}\|\nabla u\|_{L^6}+\|\rho^\frac{1-\delta}{2}\|_{L^\infty}\|\nabla u\|_{L^{3}}\|\nabla u\|_{L^6}+\|\nabla^2\rho^{\gamma-\frac{1+\delta}{2}}\|_{L^{2}}\\
&\quad+\|\rho^\frac{1-\delta}{2}\|_{L^\infty}\|u\|_{L^\infty}\|\nabla^2u\|_{L^{2}}
+\|\nabla^2\rho^\frac{\delta-1}{2}\|_{L^{2}}\|\nabla u\|_{L^\infty}
+\|\nabla\rho^\frac{\delta-1}{2}\|_{L^{6}}\|\nabla^2u\|_{L^3}\Big)\\
&\leq C\Big(\|\rho\|_{L^\infty}^{3(1-\delta)/4}\|\nabla\rho^\frac{\delta-1}{2}\|_{L^{6}}
\|\rho^\frac{1-\delta}{2}u_t\|_{L^2}^{1/2}\|\nabla u_t\|_{L^2}^{1/2}+\|\rho\|_{L^\infty}^{(1-\delta)/2}\|\nabla u_t\|_{L^2}\\
&\quad+\|\rho\|_{L^\infty}^{1-\delta}\|\nabla\rho^\frac{\delta-1}{2}\|_{L^{6}}\|\nabla u\|_{H^1}^2+\|\rho\|_{L^\infty}^{(1-\delta)/2}\|\nabla u\|_{H^1}^2+\|\nabla^2\rho^{\gamma-\frac{1+\delta}{2}}\|_{L^{2}}\\
&\quad+\|\nabla^2\rho^\frac{\delta-1}{2}\|_{L^{2}}\|\nabla^2u\|_{L^2}^{1/2}\|\nabla^2u\|_{H^1}^{1/2}
+\|\nabla\rho^\frac{\delta-1}{2}\|_{L^{6}}\|\nabla^2u\|_{L^2}^{1/2}\|\nabla^2u\|_{H^1}^{1/2}\Big)\\
&\leq C\mathcal{E}^\beta\bigl(1+\|\nabla u_t\|_{L^2}\bigr).
\end{split}
\end{equation}
Therefore, we can obtain that
\begin{align}\label{98y}
\|w\|_{D^{1}}^2&=\frac{1}{(2\mu+\lambda)^2}\|\rho^\frac{\delta-1}{2}(\mathcal{L}u+\mu\curl\curl u)\|_{D^{1}}^2\nonumber\\
&\leq C\|\nabla(\rho^\frac{\delta-1}{2}\mathcal{L} u)\|_{L^2}^2+C\|\nabla\left(\rho^\frac{\delta-1}{2}\curl\curl u\right)\|_{L^2}^2\nonumber\\
&\leq C\|\nabla(\rho^\frac{\delta-1}{2}\mathcal{L} u)\|_{L^2}^2+C\|\div\left(\rho^\frac{\delta-1}{2}\curl\curl u\right)\|_{L^2}^2+C\|\curl\left(\rho^\frac{\delta-1}{2}\curl\curl u\right)\|_{L^2}^2\nonumber\\
&\quad+C\|\rho^\frac{\delta-1}{2}\curl\curl u\cdot n\|_{H^{1/2}(\partial\Omega)}^2+C\|\rho^\frac{\delta-1}{2}\curl\curl u\|_{L^2}^2\\
&\leq C\|\nabla(\rho^\frac{\delta-1}{2}\mathcal{L} u)\|_{L^2}^2+C\|\nabla\rho^\frac{\delta-1}{2}\|_{L^6}^2\|\nabla^2 u\|_{L^3}^2+C\|\curl\left(\rho^\frac{\delta-1}{2}\curl\curl u\right)\|_{L^2}^2\nonumber\\
&\quad+C\mathcal{E}^\beta\|\curl\curl u\|_{H^1}^2+C\|\rho^\frac{\delta-1}{2}\curl\curl u\|_{L^2}^2\nonumber\\
&\leq C\mathcal{E}^\beta\bigl(1+\|\nabla u_t\|_{L^2}^2\bigr)+C\|\rho^\frac{\delta-1}{2}\mathcal{L}u\|_{L^2}^2,\nonumber
\end{align}
owing to Lemma \ref{l22}, \eqref{rho_inf}, \eqref{curlu}, and \eqref{114}. 
Additionally, Lemma \ref{G-N} combined with \eqref{114} and \eqref{98y} implies
\begin{align}\label{117}
\|w\|_{L^6}\leq C\| w\|_{H^1}\leq C\|\rho^\frac{\delta-1}{2}\mathcal{L}u\|_{L^2}+ C\mathcal{E}^\beta\bigl(1+\|\nabla u_t\|_{L^2}\bigr).
\end{align}

Defined $\varPhi\triangleq\rho^\frac{\delta-1}{2}$, one gets from $\eqref{9}_1$ multiplied by $(\frac{\delta-1}{2})\rho^{\frac{\delta-1}{2}-1}$ that 
\begin{equation}\label{38}
\varPhi_t+u\cdot\nabla\varPhi+\frac{\delta-1}{2}\varPhi\mathrm{div}u=0.
\end{equation}
 Integrating \eqref{38} multiplied  by $|\varPhi|^{4}\varPhi$ over $\Omega_0$,  one gets
\begin{equation}
	\begin{split}\label{popza}
		\frac{d}{dt}\|\varPhi\|_{L^6(\Omega_0)}^6
		\leq& C\int_{\Omega_0}\bigl(|u||\nabla\varPhi||\varPhi|^5
		+|\varPhi|^6|\mathrm{div}u|\bigr)dx\\
		\leq&C\bigl(\|u\|_{L^\infty}\|\nabla\varPhi\|_{L^6}\|\varPhi\|_{L^6(\Omega_0)}^5
		+\|\varPhi\|_{L^6(\Omega_0)}^{6}\|\nabla u\|_{L^\infty}\bigr).
	\end{split}
\end{equation}
Now, operating  $\nabla$ to \eqref{38}   and multiplying the resulting equality by $|\nabla\varPhi|^{4}\nabla\varPhi$,  one obtains after integration by parts over $\Omega_R$ that 
\begin{equation}
	\begin{split}\label{pop}
		\frac{d}{dt}\|\nabla\varPhi\|_{L^6}^6
		\leq& C\int\bigl(|\nabla\varPhi|^6|\nabla u|
		+|\nabla\varPhi|^5|\rho^\frac{\delta-1}{2}\nabla\mathrm{div}u|\bigr)dx\\
		\leq&C\bigl(\|\nabla u\|_{L^\infty}\|\nabla\varPhi\|_{L^6}^6
		+\|\nabla\varPhi\|_{L^6}^{5}\|\rho^\frac{\delta-1}{2}\nabla\mathrm{div}u\|_{L^6}\bigr).
	\end{split}
\end{equation}
Then, operating  $\nabla^2$ to \eqref{38}  and multiplying the resulting equality by $\nabla^2\varPhi$, integrating it over $\Omega_R$, it holds that
\begin{equation}\label{pop1}
\begin{split}
\frac{d}{dt}\|\nabla^2\varPhi\|_{L^{2}}^2
\leq& C\int\bigl(|\nabla u||\nabla^{2}\varPhi|^2+|\nabla^2u||\nabla\varPhi||\nabla^2\varPhi|+|\nabla^2\varPhi||\nabla(\rho^\frac{\delta-1}{2}\nabla\mathrm{div}u)|\bigr)dx\\
\leq&C\bigl(\|\nabla u\|_{L^\infty}\|\nabla^2\varPhi\|_{L^{2}}^{2}+\|\nabla^2u\|_{L^3}\|\nabla\varPhi\|_{L^{6}}\|\nabla^2\varPhi\|_{L^{2}}\\
&\quad\quad +\|\nabla^2\varPhi\|_{L^{2}}\|\nabla(\rho^\frac{\delta-1}{2}\nabla\mathrm{div}u)\|_{L^{2}}\bigr).
\end{split}
\end{equation}

Finally, we deduce from \eqref{pop}-\eqref{pop1} and \eqref{g3u} that
\begin{equation}\nonumber
\begin{split}
&\frac{d}{dt}\left(\|\varPhi\|_{L^6(\Omega_0)}+\|\nabla\varPhi\|_{L^6\cap D^{1}}\right)\\
&\leq C\|\nabla u\|_{H^2}\left(\|\varPhi\|_{L^6(\Omega_0)}+\|\nabla\varPhi\|_{L^6\cap D^{1}}\right)+C\|\rho^\frac{\delta-1}{2}\nabla\mathrm{div}u\|_{L^6\cap D^{1}}\\
&\leq C\mathcal{E}^\beta+C\|\nabla u_t\|_{L^2}^2+C\|\rho^\frac{\delta-1}{2}\mathcal{L}u\|_{L^2}^2.
\end{split}
\end{equation}
Integrating the above inequality over $(0,t)$, using \eqref{e313} and \eqref{e310}, we thus obtain \eqref{e343} and complete the proof of Lemma \ref{L3.6}.
\end{proof}

Now, Proposition \ref{p1} is a direct consequence of Lemmas \ref{L3.1}-\ref{L3.6}.

\begin{solution}
First,   direct calculations imply that
\begin{equation}
    \begin{split}\label{dsxc}
\sup_{0\leq s\leq t}\|\nabla u\|_{L^2}^2&\leq C\sup_{0\leq s\leq t}\big(\|\curl u\|_{L^2}^2+\|\div u\|_{L^2}^2+\|u\|_{L^2(\Omega_0)}^2\big)\\
&\leq C\exp\bigg(C\int_0^t\mathcal{E}^\beta ds\bigg)+C\sup_{0\leq s\leq t}\big(\|\rho^\frac{\delta-1}{2}\|_{L^\infty(\Omega_0)}^2\|\rho^\frac{1-\delta}{2}u\|_{L^2(\Omega_0)}^2\big)\\
&\leq C\exp\bigg(C\int_0^t\mathcal{E}^\beta ds\bigg),
    \end{split}
\end{equation}
where one has used Lemma \ref{l22}, \eqref{e34}, \eqref{e310}, and \eqref{e343}.

Therefore, it follows from \eqref{e34}, \eqref{e310}, \eqref{e313}, \eqref{e336}, \eqref{e343}, and \eqref{dsxc} that for all $t\in(0,T_1]$,
\begin{equation}
\mathcal{E}^2(t)\leq  CC_0^\beta+C\int_0^t\mathcal{E}^\beta ds+C\exp\bigg(C\int_0^t\mathcal{E}^\beta ds\bigg)\leq C\exp\bigg(C\int_0^t\mathcal{E}^\beta ds\bigg).
\end{equation}
Hence, the standard arguments yield that for $M\triangleq Ce^C$ and $T_0\triangleq \min\{T_1,(CM^\beta)^{-1}\}$,
\begin{equation}\nonumber
\sup_{0\leq t\leq T_0}\mathcal{E}(t)\leq M.
\end{equation}
which together with \eqref{e34}, \eqref{e310}, and \eqref{e313} gives \eqref{e32}. The proof of Proposition \ref{p1} is  completed.
\end{solution}

\begin{remark}
%Similar to the proof of \cite[Lemma3.4]{Lili2025}, w
We can also deduce from Lemma \ref{G-N}, \eqref{114}, \eqref{e336}, and \eqref{117} that  for all $t\in(0,T_0]$,
\begin{equation}\label{3.49}
\begin{split}
\sup\limits_{0\leq s\leq t}(\|u\|_{L^6}+\|\nabla^2 u\|_{L^2})
+\int_0^{t}\Big(\|u_t\|_{L^6}^2+\|\nabla^3 u\|_{L^2}^2+\|\rho^\frac{\delta-1}{2}\nabla\mathrm{div}u\|_{ L^6\cap D^{1}}^2\Big)ds\leq C.
\end{split}
\end{equation}
\end{remark}

\section{Proof of Theorem 1.1}
Letting $(\rho_0, u_0)$ be as in Theorem 1.1, since the vacuum only appears in the far-field, we first construct the  approximate smooth $\rho^R_0\in C^\infty(\Omega)$ satisfying that $\rho^R_0>0$ and 
\begin{equation}
\left\{ \begin{array}{l}
(\rho^R_0)^\frac{\delta-1}{2}\rightarrow \rho_0^\frac{\delta-1}{2} \in L^6(\Omega_0),\\
(\rho^R_0)^{\gamma-\frac{1+\delta}{2}}\rightarrow \rho_0^{\gamma-\frac{1+\delta}{2}}\ in\ D_0^1(\Omega)\cap D^2(\Omega),\\
\nabla(\rho^R_0)^\frac{\delta-1}{2}\rightarrow \nabla\rho_0^\frac{\delta-1}{2}\ in\ D_0^{1}(\Omega),
       \end{array} \right. ~~~~~~~~~\mbox{as}~~~~~~~~~R\rightarrow\infty.
\end{equation}
%as $R\rightarrow\infty$. 
%Actually, we can directly smooth $\rho_0$ since the vacuum only appears in the far-field.
Next, we consider the unique strong solution of the following Lam\'{e} system
\begin{equation}
\left\{ \begin{array}{l}
\mathcal{L}u_0^R=-(\rho^R_0)^{1-\delta}u_0^R+(\rho^R_0)^\frac{1-\delta}{2}(\rho_0^\frac{1-\delta}{2}u_0+g)\ast j_\frac{1}{R},\\
u_0^R\cdot n=0,\ \mathrm{curl}u_0^R\times n=-Au,\ x\in\partial \Omega,\\
u_0^R\cdot n=0,\ \mathrm{curl}u_0^R\times n=0,\ x\in\partial B_R.
       \end{array} \right.
\end{equation}
Following the similar arguments as \cite{13}, it can be proved that
\begin{equation}
\lim_{R\rightarrow\infty}(\|u_0^R-u_0\|_{D^1\cap D^2}+\|(\rho^R_0)^\frac{1-\delta}{2}u_0^R-\rho_0^\frac{1-\delta}{2}u_0\|_{L^2})=0.
\end{equation}
%For details, see \cite{13}. 
Then, Lemma \ref{l25} implies that  the IBVP \eqref{9} with the initial data $(\rho^R_0,u_0^R)$ has a classical solution $(\rho^R,u^R)$ on $\Omega_R\times[0,T_R]$. Moreover, Proposition \ref{p1} shows that  there exists a $T_0$ independent of $R$ such that both \eqref{e32} and \eqref{3.49} hold  for $(\rho^R,u^R)$. Extending $(\rho^R,u^R,\phi^R:=\nabla(\rho^R)^\frac{\delta-1}{2})$ by zero on $\Omega\setminus \Omega_R$ and denoting
\begin{equation}
\tilde{\rho}^R=\rho^R\varphi_R^\nu,
~~\tilde{\phi}^R=\nabla(\rho^R)^\frac{\delta-1}{2}\varphi_R,~~w^R=u^R\varphi_R,
\end{equation}
with $\nu=\frac{2}{2\gamma-1-\delta}$ and $\varphi_R\in C_0^\infty(B_R)$ satisfying 
\begin{equation}\label{phi_R}
    0\leq\varphi_R\leq 1,~~\varphi_R(x)\leq 1 ~\text{for}~\vert x\vert\leq R/2,~~\vert\nabla^k\varphi_R\vert\leq CR^{-k}(k=1,2,3),
\end{equation}
for $R> 2R_0+1$, we thus deduce from Proposition \ref{p1}, \eqref{3.49}, and \eqref{phi_R} that
\begin{equation}\label{4.6}
\begin{split}
&\sup_{0\leq t\leq T_0}(\|(\tilde{\rho}^R)^\frac{1-\delta}{2}w^R\|_{L^2(\Omega)}
		+\|\nabla w^R\|_{L^2(\Omega)}
		+\|(\tilde{\rho}^R)^\frac{1-\delta}{2}w^R_t\|_{L^2(\Omega)})\\
  &\leq C+C\sup_{0\leq t\leq T_0}(\|\nabla u^R\|_{L^2(\Omega_R)}
  +\|u^R\|_{L^6(\Omega_R)}\|\nabla \varphi_R\|_{L^3(\Omega_R)})\\
  &\leq  C,
\end{split}
\end{equation}
and
\begin{equation}
\begin{split}
    &\sup_{0\leq t\leq T_0}(\|w^R\|_{L^6(\Omega)}
		+\|\nabla^2w^R\|_{L^2(\Omega)}
		)\\
  &\leq C+C\sup_{0\leq t\leq T_0}(\|\nabla^2 u^R\|_{L^2(\Omega_R)}
  +\frac{1}{R}\|\nabla u^R\|_{L^2(\Omega_R)}
  +\|u^R\|_{L^6(\Omega_R)}\|\nabla^2 \varphi_R\|_{L^3(\Omega_R)})\\
  &\leq  C.
\end{split}
\end{equation}

Next, the straightforward calculations yield that
\begin{equation}
\begin{split}
    &\sup_{0\leq t\leq T_0}\|(\tilde{\rho}^R)^{\gamma-\frac{1+\delta}{2}}\|_{W^{1,6}\cap D^1\cap D^2(\Omega)}\\
  &\leq C+C\sup_{0\leq t\leq T_0}(\|\nabla (\rho^R)^{\gamma-\frac{1+\delta}{2}}\|_{L^2\cap L^6(\Omega_R)}
  +\|(\rho^R)^{\gamma-\frac{1+\delta}{2}}\|_{W^{1,6}(\Omega_R)}\|\nabla \varphi_R\|_{L^3\cap L^6(\Omega_R)})\\
  &\quad+C\sup_{0\leq t\leq T_0}(\|\nabla^2(\rho^R)^{\gamma-\frac{1+\delta}{2}}\|_{L^2(\Omega_R)}+\|(\rho^R)^{\gamma-\frac{1+\delta}{2}}\|_{L^6(\Omega_R)}\|\nabla \varphi_R\|_{L^3(\Omega_R)})\\
  &\leq  C,
\end{split}
\end{equation}
and that
\begin{equation}
\begin{split}
    &\sup_{0\leq t\leq T_0}\|\tilde{\phi}^R\|_{L^{6}\cap D^{1,2}(\Omega)}\\
  &\leq C+C\sup_{0\leq t\leq T_0}(\|\nabla^2 (\rho^R)^\frac{\delta-1}{2}\|_{L^{2}(\Omega_R)}
  +\|\nabla (\rho^R)^\frac{\delta-1}{2}\|_{L^{6}(\Omega_R)}\|\nabla \varphi_R\|_{L^3(\Omega_R)})\\
  &\leq  C.
\end{split}
\end{equation}

Similarly, it follows from Proposition \ref{p1} and \eqref{3.49}  that
%\begin{equation}
%    \int_0^{T^\ast}\bigl(\|\nabla w^R_t\|_{L^2(\Omega)}\bigr)dt
%  \leq C,
%\end{equation}
\begin{equation}
    \int_0^{T_0}\bigl(\|\nabla w^R_t\|_{L^2(\Omega)}+\| w^R_t\|_{L^6(\Omega)}^2+\|\nabla^3 w^R\|_{L^2(\Omega)}^2\bigr)dt
  \leq C,
\end{equation}
and
\begin{equation}\label{4.12}
\begin{split}
    &\int_0^{T_0}\|\bigl((\tilde{\rho}^R)^{\gamma-\frac{1+\delta}{2}}\bigr)_t\|_{L^6(\Omega)}^2dt\\
    &\leq
     \int_0^{T_0}\|u\cdot\nabla(\rho^R)^{\gamma-\frac{1+\delta}{2}}\|_{L^6(\Omega_R)}^2
    +\|(\rho^R)^{\gamma-\frac{1+\delta}{2}}\mathrm{div}u\|_{L^6(\Omega_R)}^2dt\\
    &\leq\int_0^{T_0}\|(\rho^R)^{\gamma-\frac{1+\delta}{2}}\|_{W^{1,6}(\Omega_R)}^2
    \|\nabla u\|_{H^1(\Omega_R)}^2dt\\
  &\leq C,
\end{split}
\end{equation}
due to  \eqref{310}. 
With all these estimates \eqref{4.6}--\eqref{4.12} at hand, there exists a subsequence $R_j$, $R_j\rightarrow\infty$, such that $(\tilde{\rho}^{R_j},w^{R_j},\tilde{\phi}^{R_j})$ converge to some limit $(\rho,u,\phi)$ in some weak sense as follows:
\begin{align}
	\rho^{R_j}\rightarrow\rho,u^{R_j}\rightarrow u\ in\ C(\overline{\Omega_N}\times[0,T_0]),\forall\ N>2R_0+1,\label{strong}\\
	(\rho^{R_j})^{\gamma-\frac{1+\delta}{2}}\rightharpoonup \rho^{\gamma-\frac{1+\delta}{2}}\ \mathrm{weakly}^\ast\ in\ L^\infty(0,T_0; D_0^1\cap D^2(\Omega)),\\
	\tilde{\phi}^R\rightharpoonup\phi\ \mathrm{weakly}^\ast\ L^\infty(0,T_0;D_0^{1,2}(\Omega)),\\
	u^{R_j}\rightharpoonup u\ \mathrm{weakly}^\ast\ in\ L^\infty(0,T_0;D_0^1(\Omega)\cap D^2(\Omega)),\\
	\nabla^3u^{R_j}\rightharpoonup\nabla^3u\ \mathrm{weakly}\ in\ L^2((0,T_0)\times\Omega),\\
	u^{R_j}_t\rightharpoonup u_t\ \mathrm{weakly}\ in\ L^2(0,T_0;D_0^1(\Omega)).\label{weak}
\end{align}
And it is easy to check that $\phi=\nabla \rho^\frac{\delta-1}{2}$ in the weak sense since $\tilde{\phi}^R=\nabla(\tilde{\rho}^R)^\frac{\delta-1}{2}$ in $\Omega_{R/2}$ for all $R>2R_0+1$. Next, for any test function $\Psi\in C_0^\infty(\Omega\times[0,T_0])$, it follows form \eqref{strong}--\eqref{weak} that
\begin{equation}
    \begin{split}
    \lim_{j\rightarrow\infty}\int_0^{T_0}\int_{\Omega}(\rho^{R_j})^\frac{\delta-1}{2}\mathcal{L}u^{R_j}\varphi_{R_j}\Psi dx dt
        &=\lim_{j\rightarrow\infty}\int_0^{T_0}\int_{supp\Psi}(\rho^{R_j})^\frac{\delta-1}{2}\mathcal{L}u^{R_j}\varphi_{R_j}\Psi dx dt\\
        &=\lim\limits_{\substack{j\rightarrow\infty \\supp\Psi\subset B_{R_j/2}}}\int_0^{T_0}\int_{supp\Psi}(\rho^{R_j})^\frac{\delta-1}{2}\mathcal{L}u^{R_j}\varphi_{R_j}\Psi dxdt\\
        &=\lim\limits_{\substack{j\rightarrow\infty \\supp\Psi\subset B_{R_j/2}}}\int_0^{T_0}\int_{\Omega}(\tilde{\rho}^{R_j})^\frac{\delta-1}{2}\mathcal{L}w^{R_j}\Psi dx dt\\
        &=\int_0^{T_0}\int_{\Omega}\rho^\frac{\delta-1}{2}\mathcal{L}u\Psi dx dt,
    \end{split}
\end{equation}
and thus
\begin{align}
        (\rho^{R_j})^\frac{\delta-1}{2}\mathcal{L}u^{R_j}\varphi_{R_j}\rightharpoonup \rho^\frac{\delta-1}{2}\mathcal{L}u\ \mathrm{weakly}\ in\ L^2((0,T_0)\times\Omega).
\end{align}
Similarly, we can obtain
\begin{align}
	(\rho^{R_j})^\frac{\delta-1}{2}\nabla\mathrm{div}u^{R_j}\varphi_{R_j}\rightharpoonup\rho^\frac{\delta-1}{2}\nabla\mathrm{div}u\ \mathrm{weakly}\ in\ L^2((0,T_0)\times\Omega),\\
	\nabla((\rho^{R_j})^\frac{\delta-1}{2}\nabla\mathrm{div}u^{R_j})\varphi_{R_j}\rightharpoonup\nabla(\rho^\frac{\delta-1}{2}\nabla\mathrm{div}u)\ \mathrm{weakly}\ in\ L^2((0,T_0)\times\Omega).
\end{align}
Moreover, $(\rho,u)$ also satisfies   \eqref{e32} and \eqref{3.49}. Therefore, we can easily show that $(\rho,u)$ is a strong solution to the  problem 
\eqref{1}-\eqref{7} satisfying the regularity \eqref{e116}.

Now, it only remains to prove the uniqueness of the strong solutions satisfying \eqref{e116}.
Let $(\rho_1,u_1)$ and $(\rho_2,u_2)$ be two strong solutions satisfying \eqref{e116} with the same initial data. Set
$$
\begin{gathered}
\varphi_i =\rho^{\gamma-\frac{1+\delta}{2}}_i,\ h_i=\rho^\frac{1-\delta}{2}_i,\ g_i=\log\rho_i,\ i=1,2,\\
\bar{\varphi}=\varphi_1-\varphi_2,\ \bar{h}=h_1-h_2,\ \bar{g}=g_1-g_2,\ \bar{u}=u_1-u_2.
\end{gathered}
$$
Subtracting the momentum equations satisfied by $(\rho_1,u_1)$ and $(\rho_2,u_2)$ yields
\begin{equation}\label{2.96}
\begin{split}
h_1^2\bar{u}_t+h_1^2u_1\cdot\nabla\bar{u}-\mathcal{L}\bar{u}=-h_1\nabla\bar{\varphi}-\bar{h}\nabla\varphi_2-\bar{h}(h_1+h_2)(u_2)_t\\
-\bar{h}(h_1+h_2)u_2\cdot\nabla u_2-h_1^2\bar{u}\cdot\nabla u_2+\nabla\bar{g}\cdot\mathcal{S}(u_2)+\nabla g_1\cdot\mathcal{S}(\bar{u}),
\end{split}
\end{equation}
in which all the coefficients is assumed to be $1$ for simplicity. Meanwhile,  $(\overline{\varphi},\overline{h},\overline{g},\overline{u})$ meets the following initial conditions
\begin{equation*}
(\overline{\varphi},\overline{h},\overline{g},\overline{u})|_{t=0}=(0,0,0,0),
\end{equation*}
and boundary conditions
\begin{equation*}
	\overline{u}\cdot n|_{\partial\Omega}=0,\quad \mbox{curl}\overline{u}\times n|_{\partial\Omega}=-A\overline{u}.
\end{equation*}

Multiplying \eqref{2.96} by $\bar{u}$ and integrating by parts lead to
\begin{equation}
	\begin{split}\label{try1}
		&\frac{1}{2}\frac{d}{dt}\int h_1^2|\bar{u}|^2dx+\mu\|\mathrm{curl} \bar u\|_{L^2}^2+
(2\mu+\lambda)\|\mathrm{div}\bar u\|_{L^2}^2+\mu\int_{\partial\Omega}\bar u\cdot A\cdot \bar udS\\
		&\leq C\Big(\|\mathrm{div}u_1\|_{L^\infty}\|h_1\bar{u}\|_{L^2}^2
		+\|\nabla\bar{\varphi}\|_{L^2}\|h_1\bar{u}\|_{L^2}+\|\bar{h}\|_{L^3}\|\nabla\varphi_2\|_{L^2}\|\bar{u}\|_{L^6}\\
		&\quad
		+\|\bar{h}\|_{L^3}(\|h_1\bar{u}\|_{L^2}\|(u_2)_t\|_{L^6}+\|h_2(u_2)_t\|_{L^2}\|\bar{u}\|_{L^6})\\
		&\quad+\|\bar{h}\|_{L^3}\|h_1+h_2\|_{L^\infty}\|u_2\|_{L^6}\|\nabla u_2\|_{L^3}\|\bar{u}\|_{L^6}+\|\nabla u_2\|_{L^\infty}\|h_1\bar{u}\|_{L^2}^2\\
		&\quad+\|\nabla\bar{g}\|_{L^2}\|\nabla u_2\|_{L^3}\|\bar{u}\|_{L^6}
		+\|h_1\|_{L^\infty}^\frac12\|\nabla(h_1)^{-1}\|_{L^{6}}\|h_1\bar{u}\|_{L^2}^\frac12\|\nabla\bar{u}\|_{L^2}^\frac32\Big)\\
		&\leq \epsilon\|\nabla \bar u\|_{L^2}^2+CF(t)(\|h_1\bar{u}\|_{L^2}^2+\|\nabla\bar{\varphi}\|_{L^2}^2+\|\bar{h}\|_{H^1}^2+\|\nabla\bar{g}\|_{L^2}^2),
	\end{split}
\end{equation}
where $F(t)=1+\|\nabla u_1\|_{H^2}^2+\|\nabla u_2\|_{H^2}^2+\|\nabla (u_1)_t\|_{L^2}^2+\|\nabla (u_2)_t\|_{L^2}^2$.
The trace theorem implies that the boundary term in \eqref{try1} can be governed as
\begin{align}\label{trace6}
\mu\int_{\partial\Omega}\bar u\cdot A\cdot \bar udS\leq C\||u|^2\|_{W^{1,1}(\Omega_0)}  &\leq \epsilon\|\nabla \bar{u}\|_{L^2}^2+C\|h_1\bar{u}\|_{L^2}^2,
\end{align}
where we have used
\begin{equation}\label{rho_inf1}
    \|h_1^{-1}\|_{L^\infty(\Omega_0)}=\|h_1^{-1}\|_{W^{1,6}(\Omega_0)}\leq C.
\end{equation}

Subtracting the mass equations satisfied by $(\rho_1,u_1)$ and $(\rho_2,u_2)$ yields
\begin{equation}\label{2.98}
	\bar{h}_t+\nabla\bar{h}\cdot u_2+\bar{h}\mathrm{div}u_2=-\nabla h_1\cdot \bar{u}-h_1\mathrm{div}\bar{u},
\end{equation}
\begin{equation}\label{2.99}
	(\nabla\bar{g})_t+\nabla^2\bar{g}\cdot u_2+\nabla\bar{g}\cdot \nabla u_2=-\nabla^2g_1\cdot \bar{u}-\nabla g_1\cdot \nabla \bar{u}-\nabla\mathrm{div}\bar{u},
\end{equation}
\begin{equation}\label{2.100}
	\begin{split}
        (\nabla\bar{\varphi})_t+\nabla^2\bar{\varphi}\cdot u_2+\nabla\bar{\varphi}\cdot \nabla u_2+\nabla\bar{\varphi}\cdot \mathrm{div}u_2+\bar{\varphi}\nabla\mathrm{div}u_2\\
		=-\nabla^2\varphi_1\cdot \bar{u}-\nabla \varphi_1\cdot \nabla \bar{u}-\nabla\varphi_1\cdot \mathrm{div}\bar{u}-\varphi_1\nabla\mathrm{div}\bar{u},
	\end{split}
\end{equation}
where we take all the coefficients as 1 for simplicity.  Multiplying \eqref{2.98} by $\bar{h}$ and integrating by parts lead to
\begin{equation}
	\begin{split}
		\frac{d}{dt}\|\bar{h}\|_{L^2}^2\leq&\|\mathrm{div}u_2\|_{L^\infty} \|\bar{h}\|_{L^2}^2+\|h_1\|_{L^\infty}\|\nabla\bar{u}\|_{L^2}\|\bar{h}\|_{L^2}\\
		&+\|h_1\|_{L^\infty}\|\nabla(h_1)^{-1}\|_{L^{6}}\|h_1\bar{u}\|_{L^{3}}\|\bar{h}\|_{L^2}\\
		\leq& CF(t)(\|\bar{h}\|_{L^2}+\|h_1\bar{u}\|_{L^2}+\|\nabla \bar{u}\|_{L^2})\|\bar{h}\|_{L^2}.
	\end{split}
\end{equation}

Taking gradient of \eqref{2.98}, multiplying the resulting equations by $\nabla\bar{h}$ and integrating by parts lead to
\begin{equation}
	\begin{split}
		\frac{d}{dt}\|\nabla\bar{h}\|_{L^2}^2\leq&\|\nabla u_2\|_{L^\infty}\|\nabla\bar{h}\|_{L^2}^2+\|\bar{h}\|_{L^6}\|\nabla^2 u_2\|_{L^3}\|\nabla\bar{h}\|_{L^2}\\
		&+\|h_1\|_{L^\infty}^3\|\nabla(h_1)^{-1}\|_{L^{6}}^2\|\bar{u}\|_{L^{6}}\|\nabla\bar{h}\|_{L^2}\\
		&+\|h_1\|_{L^\infty}^2\|\nabla^2(h_1)^{-1}\|_{L^{2}}\|\bar{u}\|_{L^{\infty}}\|\nabla\bar{h}\|_{L^2}\\
		&+\|h_1\|_{L^\infty}^2\|\nabla(h_1)^{-1}\|_{L^{6}}\|\nabla\bar{u}\|_{L^{3}}\|\nabla\bar{h}\|_{L^2}\\
		&+\|h_1\|_{L^\infty}\|\nabla^2\bar{u}\|_{L^2}\|\nabla\bar{h}\|_{L^2}\\
		\leq& CF(t)(\|\bar{h}\|_{H^1}^2+\|h_1\bar{u}\|_{L^2}^2+\|\nabla \bar{u}\|_{L^2}^2)+\epsilon\|\nabla^2 \bar{u}\|_{L^2}^2.
	\end{split}
\end{equation}

Multiplying \eqref{2.99} by $\nabla\bar{g}$ and integrating by parts lead to
\begin{equation}
\begin{split}
	\frac{d}{dt}\|\nabla\bar{g}\|_{L^2}^2\leq&\|\nabla u_2\|_{L^\infty} \|\nabla\bar{g}\|_{L^2}^2+\|h_1\|_{L^\infty}\|\nabla^2(h_1)^{-1}\|_{L^{2}}\|\bar{u}\|_{L^\infty}\|\nabla\bar{g}\|_{L^2}\\
	&+\|h_1\|_{L^\infty}^2\|\nabla(h_1)^{-1}\|_{L^{6}}^2\|\bar{u}\|_{L^{6}}\|\nabla\bar{g}\|_{L^2}\\
	&+\|h_1\|_{L^\infty}\|\nabla(h_1)^{-1}\|_{L^{6}}\|\nabla\bar{u}\|_{L^{3}}\|\nabla\bar{g}\|_{L^2}+\|\nabla^2\bar{u}\|_{L^2}\|\nabla\bar{g}\|_{L^2}\\
	\leq& CF(t)(\|\nabla\bar{g}\|_{L^2}^2+\|h_1\bar{u}\|_{L^2}^2+\|\nabla \bar{u}\|_{L^2}^2)+\epsilon\|\nabla^2 \bar{u}\|_{L^2}^2.
\end{split}
\end{equation}

Multiplying \eqref{2.100} by $\nabla\bar{\varphi}$ and integrating by parts lead to
\begin{equation}
	\begin{split}
		\frac{d}{dt}\|\nabla\bar{\varphi}\|_{L^2}^2\leq&\|\nabla u_2\|_{L^\infty}\|\nabla\bar{\varphi}\|_{L^2}^2+\|\bar{\varphi}\|_{L^6}\|\nabla^2 u_2\|_{L^3}\|\nabla\bar{\varphi}\|_{L^2}\\
		&+\|\nabla^2\varphi_1\|_{L^2}\|\bar{u}\|_{L^\infty}\|\nabla\bar{\varphi}\|_{L^2}+\|\nabla\varphi_1\|_{L^6}\|\nabla\bar{u}\|_{L^3}\|\nabla\bar{\varphi}\|_{L^2}\\
		&+\|\varphi_1\|_{L^\infty}\|\nabla^2\bar{u}\|_{L^2}\|\nabla\bar{\varphi}\|_{L^2}\\
		\leq& CF(t)(\|\nabla\bar{\varphi}\|_{L^2}^2+\|h_1\bar{u}\|_{L^2}^2+\|\nabla \bar{u}\|_{L^2}^2)+\epsilon\|\nabla^2 \bar{u}\|_{L^2}^2.
	\end{split}
\end{equation}

Applying the standard $L^p$-estimate theory to \eqref{2.96} leads to
\begin{equation}\label{qaza}
\begin{split}
\|\nabla^2\bar{u}\|_{L^2}&\leq C\|\mathcal{L}\bar{u}\|_{L^2}+C\|\nabla\bar u\|_{L^2}\\
&\leq CF(t)(\|\nabla\bar{u}\|_{L^2}+\|\bar{h}\|_{H^1}+\|\nabla\bar{g}\|_{L^2})+C\|h_1\bar{u}_t\|_{L^2}+C\|\nabla\bar{\varphi}\|_{L^2}.
\end{split}
\end{equation}

Observing that  
\begin{align}\label{qaza1}
h_1^2(u_1)_t-h_2^2(u_2)_t=\frac12\Big(h_1^2\bar u_t+(h_1+h_2)\bar h(u_2)_t\Big)+\frac12\Big((h_1+h_2)\bar h(u_1)_t+h_2^2\bar u_t\Big),
\end{align}
one obtains after multiplying \eqref{2.96} by $2\bar{u}_t$ and integrating by parts that
\begin{equation}\label{e526}
\begin{split}
&\frac{d}{dt}\left(\mu\|\mathrm{curl}\bar u\|_{L^2}^2
+(2\mu+\lambda)\|\mathrm{div}\bar u\|_{L^2}^2+\mu\int_{\partial\Omega}\bar u\cdot A\cdot \bar udS\right)
+\int (h_1^2+h_2^2)|\bar{u}_t|^2dx\\
&=2\int\Big(-h_1^2u_1\cdot\nabla\bar u-h_1\nabla\bar{\varphi}-\bar h\nabla\varphi_2-\frac12(h_1+h_2)\bar{h}\big((u_2)_t+(u_1)_t\big)
\\
&\quad-\bar{h}(h_1+h_2)u_1\cdot\nabla u_2-h_1^2\bar{u}\cdot\nabla u_2+\nabla\bar{g}\cdot\mathcal{S}(u_2)+\nabla g_1\cdot\mathcal{S}(\bar{u})\Big)\cdot\bar{u}_tdx\\
&\triangleq\sum_{i=1}^8K_i.
\end{split}
\end{equation}
We estimate each $K_i(i=1,...,8)$ as follows:
\begin{equation}%\label{e527}
\begin{split}\nonumber
&K_1\leq C\|h_1\|_{L^\infty}\|u_1\|_{L^\infty}\|\nabla\bar u\|_{L^2}\|h_1\bar u_t\|_{L^2}
\leq CF(t)\|\nabla \bar u\|_{L^2}^2+\epsilon\|h_1\bar{u}_t\|_{L^2}^2,\\
&K_2\leq C\|\nabla\bar{\varphi}\|_{L^2}\|h_1\bar{u}_t\|_{L^2}\leq C\|\nabla\bar{\varphi}\|_{L^2}^2+\epsilon\|h_1\bar{u}_t\|_{L^2}^2,\\
&K_3\leq C\|\bar h\|_{L^3}\|h_2^{-1}\nabla\varphi_2\|_{L^6}\|h_2\bar u_t\|_{L^2}\leq C\|\bar{h}\|_{H^1}^2
+\epsilon\|h_1\bar{u}_t\|_{L^2}^2,\\
&K_4\leq C\|\bar h\|_{L^3}\big(\|(u_2)_t\|_{L^6}+\|(u_1)_t\|_{L^6}\big)\|(h_1+h_2) \bar u_t\|_{L^2}\\
&\ \quad \leq CF(t)\|\bar{h}\|_{H^1}^2
+\epsilon\|(h_1+h_2)\bar{u}_t\|_{L^2}^2,\\
&K_5\leq C\|\bar h\|_{L^6}\|u_1\|_{L^6}\|\nabla u_2\|_{L^6}\|(h_1+h_2)\bar u_t\|_{L^2}\\
&\ \quad \leq CF(t)\|\nabla \bar{h}\|_{L^2}^2+\epsilon\|(h_1+h_2)\bar{u}_t\|_{L^2}^2,\\
&K_6\leq C\|h_1\|_{L^\infty}\|\bar{u}\|_{L^6}\|\nabla u_2\|_{L^3}\|h_1\bar{u}_t\|_{L^2}\leq C\|\nabla u_2\|_{H^1}^2\|\nabla\bar u\|_{L^2}+\epsilon\|h_1\bar{u}_t\|_{L^2}^2,\\
&K_8\leq C\|\nabla(h_1)^{-1}\|_{L^{6}}\|\nabla\bar{u}\|_{L^{3}}\|h_1\bar{u}_t\|_{L^2}\leq C\|\nabla\bar u\|_{L^2}^2+\epsilon\|\nabla^2\bar u\|_{L^2}^2+\epsilon\|h_1\bar{u}_t\|_{L^2}^2,
\end{split}
\end{equation}
and
\begin{equation}
\begin{split}\nonumber
K_7=&\frac{d}{dt}\int \nabla\bar{g}\cdot\mathcal{S}(u_2)\cdot\bar{u}dx-\int\nabla\bar{g}_t\cdot\mathcal{S}(u_2)\cdot\bar{u}dx-\int\nabla\bar{g}\cdot\mathcal{S}(u_2)_t\cdot\bar{u}dx\\
\leq&\frac{d}{dt}\int \nabla\bar{g}\cdot\mathcal{S}(u_2)\cdot\bar{u}dx+\|\nabla\bar{g}\|_{L^2}\|\nabla u_1\|_{L^6}\|\nabla u_2\|_{L^6}\|\bar{u}\|_{L^6}\\
&+\|\nabla\bar{g}\|_{L^2}\|u_1\|_{L^\infty}\|\nabla^2 u_2\|_{L^3}\|\bar{u}\|_{L^6}+\|\nabla\bar{g}\|_{L^2}\| u_1\|_{L^\infty}\|\nabla u_2\|_{L^6}\|\nabla\bar{u}\|_{L^3}\\
&+\|\nabla\bar g\|_{L^2}\|\nabla u_1\|_{L^6}\|\nabla u_2\|_{L^6}\|\bar u\|_{L^6}
+\|\nabla^2 g_1\|_{L^{2}}\|\bar{u}\|_{L^6}^2\|\nabla u_2\|_{L^{6}}\\
&+\|\nabla g_1\|_{L^{6}}\|\nabla\bar{u}\|_{L^2}\|\nabla u_2\|_{L^{6}}\|\bar{u}\|_{L^6}
+\|\nabla^2\bar{u}\|_{L^2}\|\nabla u_2\|_{L^3}\|\bar{u}\|_{L^6}\\
&+\|\nabla\bar{g}\|_{L^2}\|\nabla(u_2)_t\|_{L^2}\|\bar{u}\|_{L^\infty}\\
\leq&\frac{d}{dt}\int \nabla\bar{g}\cdot\mathcal{S}(u_2)\cdot\bar{u}dx+CF(t)\|\nabla\bar g\|_{L^2}^2+C\|\nabla\bar u\|_{L^2}^2+\epsilon\|\nabla^2 \bar u\|_{L^2}^2.
\end{split}
\end{equation}
Substituting $K_1$--$K_8$ into \eqref{e526}  and choosing $\epsilon$ small enough, one gets
\begin{equation}
	\begin{split}\nonumber
		&\frac{d}{dt}\Big(\mu\|\mathrm{curl}\bar u\|_{L^2}^2
+(2\mu+\lambda)\|\mathrm{div}\bar u\|_{L^2}^2+K(t)\Big)+\int (h_1^2+h_2^2)|\bar{u}_t|^2dx\\
		&\leq CF(t)(\|\nabla\bar{u}\|_{L^2}^2+\|\bar{h}\|_{H^1}^2+\|h_1\bar{u}\|_{L^2}^2+\|\nabla\bar{g}\|_{L^2}^2),
	\end{split}
\end{equation}
where
\begin{equation}\nonumber
	K(t)=\mu\int_{\partial\Omega}\bar u\cdot A\cdot \bar udS-\int\nabla\bar{g}\cdot\mathcal{S}(u_2)\cdot\bar{u}dx
\end{equation}
satisfying the following estimates
\begin{equation}\nonumber
|K(t)|\leq \epsilon\|\nabla\bar{u}\|_{L^2}^2+C\|\nabla\bar{g}\|_{L^2}^2+C\|h_1\bar{u}\|_{L^2}^2.
\end{equation}

Finally, let
\begin{equation}
\begin{split}\nonumber
G(t)\triangleq&\|h_1\bar{u}\|_{L^2}^2+\|\bar{h}\|_{H^1}^2+\|\nabla\bar{g}\|_{L^2}^2
+\|\nabla\bar{\varphi}\|_{L^2}^2+\nu\big(\mu\|\mathrm{curl}\bar u\|_{L^2}^2
+(2\mu+\lambda)\|\mathrm{div}\bar u\|_{L^2}^2+K(t)\big),
\end{split}
\end{equation}
we deduce after choosing $\nu$ and $\epsilon$ small enough that
\begin{equation}\nonumber
\frac{d}{dt}G(t)+\nu\int (h_1^2+h_2^2)|\bar{u}_t|^2dx\leq CF(t)G(t),
\end{equation}
which together with Gr\"onwall's inequality yields $G(t)\equiv 0$.
The proof of Theorem 1.1 is completed.

\section{Proof of Theorem 1.2}

Let $(\rho, u)$ be a strong solution to the IBVP \eqref{1}-\eqref{7} described in Theorem \ref{T1} with the following additional initial value condition \begin{equation}\label{2601291}
 \rho_0^{\gamma-\frac{1+\delta}{2}}\in L^{6\alpha},~~~~ 
\max\{\frac{1-\delta}{2\gamma-1-\delta},\frac13\}<\alpha<1,
\end{equation}
suppose that \eqref{blowup1} is false, that is, 
\begin{equation}\label{4.1}
\lim_{T\rightarrow T^\ast}\left(\|\mathcal{D}(u)\|_{L^1(0,T;L^\infty)}+\|\rho^\frac{\delta-1}{2}\|_{L^\infty(0,T;L^6(\Omega_0))}+\|\nabla\rho^\frac{\delta-1}{2}\|_{L^\infty(0,T; D_0^{1})}\right)\leq M_0<\infty.
\end{equation}
This together with $\eqref{9}_1$ immediately yields the following $L^\infty$ bound for the density $\rho$.
\begin{lemma}\label{l51}
Assume that
\begin{equation}\label{4.2}
\int_0^T \|\mathrm{div}u\|_{L^\infty}dt\leq C,\quad 0<T<T^\ast.
\end{equation}
Then
\begin{equation}\label{qa76}
\sup_{0\leq t\leq T}\left(\|\rho^{\gamma-\frac{1+\delta}{2}}\|_{L^{6\alpha}}+\|\rho\|_{L^\infty}\right)\leq C,\quad 0<T<T^\ast.
\end{equation}
Here (and in this section) $C$ will denote a generic constant depending only on $a$, $\delta$, $\gamma$, $\mu$, $\lambda$, $A$, $R_0$, $M_0$, $T$, and the initial data.
\end{lemma}
\begin{proof}
First, it is easy to see that the continuity equation $\eqref{9}_1$ on the characteristic curve $X(s; x, t)$, which is defined by
\begin{equation}\nonumber
\left\{ \begin{array}{l}
\frac{d}{ds}X(s; x, t)= u(X(s; x, t),s),\\
X(s; x, t)=x, \end{array} \right.
\end{equation}
can be written as
\begin{equation}\nonumber
\frac{d}{ds}\rho(X(s;x,t),t)=-(\rho \mathrm{div}u)(X(s;x,t),t).
\end{equation}
Therefore, Gr\"onwall's inequality and \eqref{4.2} yield for all $x\in\Omega$
	\begin{equation}
		\begin{split}\nonumber
			\rho(x,t)\leq\sup_{x\in\Omega}\rho_0(x)\exp\bigg(\int_0^T \|\mathrm{div}u\|_{L^\infty}dt\bigg)\leq C.
		\end{split}
	\end{equation}

Next, it follows from $\eqref{9}_1$ that $\rho^{\gamma-\frac{1+\delta}{2}}$ satisfies
\begin{equation}
\rho^{\gamma-\frac{1+\delta}{2}}_t+u\cdot\nabla\rho^{\gamma-\frac{1+\delta}{2}}+({\gamma-\frac{3+\delta}{2}})\rho^{\gamma-\frac{1+\delta}{2}}\mathrm{div}u=0.
\end{equation}
Thus, the standard arguments along with \eqref{2601291} show that
$$
\sup_{t\in[0,T]}\|\rho^{\gamma-\frac{1+\delta}{2}}\|_{L^{6\alpha}}\leq C.
$$
The proof of Lemma \ref{l51} is completed.
\end{proof}

We begin with the following standard energy estimate for $(\rho,u)$.
\begin{lemma}\label{l52}
	Under the condition \eqref{4.1}, it holds that for any $T < T^\ast$,
	\begin{equation}
		\sup_{0\leq t\leq T}\|\rho^\frac{1-\delta}{2}u\|_{L^2}^2
		+\int_0^T\|\nabla u\|_{L^2}^2dt\leq C. \label{4.7}
	\end{equation}
\end{lemma}
\begin{proof}
Similar to \eqref{3.4}, the combination of \eqref{4.1} with  \eqref{qa76} gives
\begin{equation}
\begin{split} \label{4.8}
&\frac{1}{2}\frac{d}{dt}\|\rho^\frac{1-\delta}{2}u\|_{L^2}^2+\mu\|\curl u\|_{L^2}^2+
(2\mu+\lambda)\|\mathrm{div}u\|_{L^2}^2+\mu\int_{\partial\Omega}u\cdot A\cdot udS \\
&\leq \int\left(C|\rho|^{1-\delta}|\mathrm{div}u||u|^2-\frac{2a\delta}{2\gamma-1-\delta}\nabla\rho^{\gamma-\frac{1+\delta}{2}}(\rho^\frac{1-\delta}{2}u)
+C|\nabla\rho^\frac{\delta-1}{2}||\rho^\frac{1-\delta}{2}u||\nabla u|\right)dx\\
&\leq C\Big(\|\mathrm{div} u\|_{L^\infty}\|\rho^\frac{1-\delta}{2}u\|_{L^2}^2
+\|\rho^{\gamma-\frac{1+\delta}{2}}\|_{L^{6\alpha}}\|\nabla(\rho^\frac{1-\delta}{2}u)\|_{L^\frac{6\alpha}{6\alpha-1}}\\
&\quad\quad+\|\rho^\frac{1-\delta}{4}\|_{L^\infty}\|\nabla\rho^\frac{\delta-1}{2}\|_{L^{6}}\|\nabla u\|_{L^2}\|u\|_{L^{6}}^{1/2}\|\rho^\frac{1-\delta}{2}u\|_{L^{2}}^{1/2}\Big)\\
&\leq \epsilon\|\nabla u\|_{L^2}^2+C(\|\mathcal{D}(u)\|_{L^\infty}+1)\|\rho^\frac{1-\delta}{2}u\|_{L^2}^2+C,
\end{split}
\end{equation}
where we have used
\begin{align}\label{ygt1}
%&\int\nabla\rho^{\gamma-\frac{1+\delta}{2}}(\rho^\frac{1-\delta}{2}u)dx\\
&\|\nabla(\rho^\frac{1-\delta}{2}u)\|_{L^\frac{6\alpha}{6\alpha-1}}\nonumber\\
&\leq C\|\rho^\frac{1-\delta}{2}u\|_{L^2}^{1-\theta}\|\rho^\frac{1-\delta}{2}u\|_{H^1}^{\theta}\nonumber\\
&\leq C\|\rho^\frac{1-\delta}{2}u\|_{L^2}^{1-\theta}\left(\|\rho^\frac{1-\delta}{2}u\|_{L^2}^{\theta}+\|\rho^\frac{3(1-\delta)}{4}\|_{L^\infty}^{\theta}
\|\nabla\rho^\frac{\delta-1}{2}\|_{L^6}^{\theta}\|\rho^\frac{1-\delta}{2}u\|_{L^2}^{\theta/2}\|u\|_{L^6}^{\theta/2}+
\|\rho^\frac{1-\delta}{2}\|_{L^\infty}^{\theta}\|\nabla u\|_{L^2}^{\theta}\right)\nonumber\\
&\leq \epsilon\|\nabla u\|_{L^2}^2+C\|\rho^\frac{1-\delta}{2}u\|_{L^2}^2+C
\end{align}
 due to $0<\theta<1$  degenerated from  $\max\{\frac{1-\delta}{2\gamma-1-\delta},\frac13\}<\alpha<1$. Using the trace theorem, the boundary term in \eqref{4.8} can be governed as
\begin{align}\label{trace14}
\int_{\partial\Omega}u\cdot A\cdot udS\leq C\||u|^2\|_{H^1(\Omega_0)}  \leq C\|\nabla u\|_{L^2}^2+C\|\rho^\frac{1-\delta}{2}u\|_{L^2}^2,
\end{align}
where we have used
\begin{equation}\label{rho_inf33}
    \|\rho^\frac{\delta-1}{2}\|_{L^\infty(\Omega_0)}=\|\rho^\frac{\delta-1}{2}\|_{W^{1,6}(\Omega_0)}\leq C.
\end{equation}

It follows from Lemma \ref{l22}, \eqref{4.8}, \eqref{trace14}, and Gr\"onwall's inequality that
	\begin{equation}
		\begin{split}
\sup_{t\in[0,T]}\|\rho^\frac{1-\delta}{2}u\|_{L^2}^2+\int_0^T\|\nabla u\|_{L^2}^2dt\leq C.
		\end{split}
	\end{equation}
 The proof of Lemma \ref{l52} is completed.
\end{proof}

The key estimates on $\nabla\rho$ and $\nabla u$ will be given in the following lemma.
\begin{lemma}\label{l522}
	Under the condition \eqref{4.1}, it holds that for any $T < T^\ast$,
	\begin{equation}
		\sup_{0\leq t\leq T}(\|\nabla\rho^{\gamma-\frac{1+\delta}{2}}\|_{L^2}^2+\|\nabla u\|_{L^2}^2)
		+\int_0^T(\|\rho^\frac{\delta-1}{2}\mathcal{L}u\|_{L^2}^2+\|\nabla^2 u\|_{L^2}^2)dt\leq C. \label{4.77}
	\end{equation}
\end{lemma}
\begin{proof}
First, the standard  regularity estimate for the Lam\'e operator (see \cite{11} for instance) and Lemma \ref{l51} yield
\begin{equation}
\|\nabla^2 u\|_{L^2}\leq C\|\mathcal{L}u\|_{L^2}+C\|\nabla u\|_{L^2}\leq C\|\rho^\frac{\delta-1}{2}\mathcal{L}u\|_{L^2}+C\|\nabla u\|_{L^2}.\label{4.16}
\end{equation}

Next, multiplying $\eqref{9}_2$  by $-2\rho^{\delta-1}\mathcal{L}u$,  one obtains after integrating the resulting equation over $\Omega$  by parts that
\begin{equation}\label{4.9}
\begin{split}
&\frac{d}{dt}\left(\mu\|\curl u\|_{L^2}^2+(2\mu+\lambda)\|\mathrm{div}u\|_{L^2}^2+\mu\int_{\partial\Omega} u\cdot A\cdot udS\right)
+2\|\rho^\frac{\delta-1}{2}\mathcal{L}u\|_{L^2}^2\\
&\leq (2\mu+\lambda)\int u\cdot\nabla u\cdot\nabla \mathrm{div}u dx-\mu\int u\cdot\nabla u\cdot\mathrm{curl}\mathrm{curl}u dx\\
&\quad+C\int\left(|\nabla\rho^{\gamma-\frac{1+\delta}{2}}||\rho^\frac{\delta-1}{2}\mathcal{L}u|
+|\nabla\rho^\frac{\delta-1}{2}||\nabla u||\rho^\frac{\delta-1}{2}\mathcal{L}u|\right)dx.
\end{split}
\end{equation}
Similar as those in  \cite{14}, one gets after integration by parts that
\begin{equation}
\begin{split}\label{4.10}
\int u\cdot\nabla u\cdot\nabla \mathrm{div}u dx&=\int_{\partial\Omega} u\cdot\nabla u\cdot n\div udS-\int\nabla u:\nabla^\bot u \mathrm{div}u dx+\frac{1}{2}\int(\mathrm{div}u)^3dx\\
&\leq \int_{\partial\Omega} u\cdot\nabla n\cdot u\div udS+C\|\mathcal{D}(u)\|_{L^\infty}\|\nabla u\|_{L^2}^2\\
&\leq C\||u|^2|\div u|\|_{W^{1,1}(\Omega_0)}+C\|\mathcal{D}(u)\|_{L^\infty}\|\nabla u\|_{L^2}^2\\
&\leq C\|\nabla u\|_{L^2}^2\|\nabla u\|_{H^1}+C\|\mathcal{D}(u)\|_{L^\infty}\|\nabla u\|_{L^2}^2.
\end{split}
\end{equation}
Using the following  facts
$$u\times\mathrm{curl}u=\frac{1}{2}\nabla(|\nabla u|^2)-u\cdot\nabla u,$$ and $$\nabla\times(a\times b)=(b\cdot\nabla)a-(a\cdot\nabla)b+(\mathrm{div}b)a-(\mathrm{div}a)b,$$ 
we have
\begin{equation}
\begin{split}
&\int (u\cdot\nabla) u\cdot\mathrm{curl}\mathrm{curl}u dx\\
&=\int_{\partial\Omega} u\cdot\nabla u\cdot A\cdot udS-\int \mathrm{curl}u\cdot\nabla\times\big((u\cdot\nabla) u\big)dx\\
&\leq C\||u\cdot\nabla u||u|\|_{W^{1,1}(\Omega_0)}+\int \mathrm{curl}u\cdot\nabla\times(u\times \mathrm{curl}u)dx\\
&\leq C\||u|^2|\nabla u|\|_{W^{1,1}(\Omega_0)}+\int\mathrm{curl}u\cdot\mathcal{D}(u)\cdot\mathrm{curl} udx-\frac{1}{2}\int|\mathrm{curl}u|^2\mathrm{div}u dx\\
&\leq C\|\nabla u\|_{L^2}^2\|\nabla u\|_{H^1}+C\|\mathcal{D}(u)\|_{L^\infty}\|\nabla u\|_{L^2}^2,
\end{split}
\end{equation}
which together with \eqref{4.1} derives
	\begin{equation}
		\begin{split}\label{bgy}
&\int\left(|\nabla\rho^{\gamma-\frac{1+\delta}{2}}||\rho^\frac{\delta-1}{2}\mathcal{L}u|+|\nabla\rho^\frac{\delta-1}{2}||\nabla u||\rho^\frac{\delta-1}{2}\mathcal{L}u|\right)dx\\
			&\leq \epsilon\|\rho^\frac{\delta-1}{2}\mathcal{L}u\|_{L^2}^2+\epsilon\|\nabla^2u\|_{L^2}^2+C(\|\nabla u\|_{L^2}^2+\|\nabla\rho^{\gamma-\frac{1+\delta}{2}}\|_{L^2}^2).
		\end{split}
	\end{equation}
Substituting \eqref{4.10}--\eqref{bgy} into \eqref{4.9} and choosing $\varepsilon$ suitably small gives that
\begin{equation}
\begin{split}\label{qaqa55}
&\frac{d}{dt}\left(\mu\|\curl u\|_{L^2}^2
+(2\mu+\lambda)\|\mathrm{div}u\|_{L^2}^2+\mu\int_{\partial\Omega} u\cdot A\cdot udS\right)
+\frac{3}{2}\|\rho^\frac{\delta-1}{2}\mathcal{L}u\|_{L^2}^2\\
&\leq C\left(\|\mathcal{D}(u)\|_{L^\infty}+\|\nabla u\|_{L^2}^2+1\right)\|\nabla u\|_{L^2}^2+C\|\nabla u\|_{L^2}^2+C\|\nabla\rho^{\gamma-\frac{1+\delta}{2}}\|_{L^2}^2.
\end{split}
\end{equation}

Denoting $$\varphi\triangleq\rho^{\gamma-\frac{1+\delta}{2}},~~~~\zeta\triangleq\gamma-\frac{1+\delta}{2},$$ 
operating $\nabla$ to \eqref{310} and multiplying the resulting equation by $\nabla\varphi$ lead to
\begin{equation}
\begin{split}
\frac{1}{2}\partial_t|\nabla\varphi|^2=&-\nabla\varphi\cdot\nabla u\cdot\nabla\varphi
-u\cdot\nabla(|\nabla\varphi|^2)-\zeta|\nabla\varphi|^2\mathrm{div}u-\zeta\varphi\nabla\varphi\cdot\nabla \mathrm{div}u\\
=&-\nabla\varphi\cdot\mathcal{D}(u)\cdot\nabla\varphi-u\cdot\nabla(|\nabla\varphi|^2) \label{4.14}
-\zeta|\nabla\varphi|^2\mathrm{div}u-\zeta\varphi\nabla\varphi\cdot\nabla \mathrm{div}u.
\end{split}
\end{equation}
Integrating \eqref{4.14} over $\Omega$, one deduces from \eqref{qa76} that 
\begin{equation}
\frac{1}{2}\frac{d}{dt}\|\nabla\varphi\|_{L^2}^2\leq C\|\mathcal{D}(u)\|_{L^\infty}\|\nabla\varphi\|_{L^2}^2
+\epsilon\|\nabla^2 u\|_{L^2}^2+C(\epsilon)\|\nabla\varphi\|_{L^2}^2. \label{4.15}
\end{equation}

Finally, by viture of  \eqref{qaqa55}, \eqref{4.15}, and \eqref{4.16}, we obtain  after choosing $\varepsilon$ suitably small that 
\begin{equation}
\begin{split}\nonumber
&\frac{d}{dt}\left(\mu\|\curl u\|_{L^2}^2
+(2\mu+\lambda)\|\mathrm{div}u\|_{L^2}^2+\|\nabla\rho^{\gamma-\frac{1+\delta}{2}}\|_{L^2}^2+\mu\int_{\partial\Omega} u\cdot A\cdot udS\right)
+\|\rho^\frac{\delta-1}{2}\mathcal{L}u\|_{L^2}^2\\
&\leq C\big(\|\mathcal{D}(u)\|_{L^\infty}+\|\nabla u\|_{L^2}^2+1\big)\left(\|\nabla u\|_{L^2}^2+\|\nabla\rho^{\gamma-\frac{1+\delta}{2}}\|_{L^2}^2\right)+C\|\nabla u\|_{L^2}^2,
\end{split}
\end{equation}
which combined with Gr\"onwall's inequality, \eqref{4.1}, \eqref{trace14},  and \eqref{4.16} gives \eqref{4.77}, and finish the proof of Lemma \ref{l522}.
\end{proof}

\begin{lemma}\label{l53}
Under the condition \eqref{4.1}, it holds that
\begin{equation}\label{4.188}
\sup_{0\leq t\leq T}(\|\rho^\frac{1-\delta}{2}u_t\|_{L^2}^2+\|\nabla u\|_{H^1}^2)
+\int_0^T\|\nabla u_t\|_{L^2}^2dt\leq C.
\end{equation}
\end{lemma}
\begin{proof}
First, it deduces from Lemma \ref{elliptic_estimate}, \eqref{e38}, \eqref{4.1}, \eqref{qa76}, \eqref{4.7}, \eqref{4.77}, and Gagliardo-Nirenberg inequality that
\begin{equation}
\begin{split}\label{4.22}
&\|\nabla^2 u\|_{L^2}\leq C\|\mathcal{L}u\|_{L^2}+\|\nabla u\|_{L^2}\\
&\leq C(\|\rho^{1-\delta}u_t\|_{L^2}+\|\rho^{1-\delta}u\cdot\nabla u\|_{L^2}
+\|\nabla\rho^{\gamma-\delta}\|_{L^2}+\|\nabla\log\rho\cdot S(u)\|_{L^2}+1)\\
&\leq C(\|\rho^\frac{1-\delta}{2}u_t\|_{L^2}+\|u\|_{L^6}\|\nabla u\|_{L^3}+\|\nabla\rho^{\gamma-\frac{1+\delta}{2}}\|_{L^2}
+\|\nabla\rho^\frac{\delta-1}{2}\|_{L^{6}}\|\nabla u\|_{L^{3}}+1)\\
&\leq \frac{1}{2}\|\nabla^2 u\|_{L^2}+C\|\rho^\frac{1-\delta}{2}u_t\|_{L^2}+C\\
&\leq C\|\rho^\frac{1-\delta}{2}u_t\|_{L^2}+C.
\end{split}
\end{equation}

Next, it follows from \eqref{qa98},  \eqref{qa981}, \eqref{4.1}, \eqref{qa76}, \eqref{4.7}, \eqref{4.77}, and Gagliardo-Nirenberg inequality that
\begin{gather}
\|(\rho^{\gamma-\delta})_t\|_{L^2}\leq\|\rho\|_{L^\infty}^{(1-\delta)/2}\|u\|_{L^\infty}\|\nabla\rho^{\gamma-\frac{1+\delta}{2}}\|_{L^2}+\|\rho\|_{L^\infty}^{\gamma-\delta}\|\nabla u\|_{L^2}\leq C\|\nabla^2 u\|_{L^2}^{1/2}+C,\label{qaz29}
\end{gather}
and
\begin{equation}
\begin{split}\label{qaz49}
\|\nabla(\log\rho)_t\|_{L^2}
&\leq C\bigl(\|\rho\|_{L^\infty}^{1-\delta}\|\nabla\rho^\frac{\delta-1}{2}\|_{L^{6}}^2\|u\|_{L^{6}}
+\|\rho\|_{L^\infty}^{(1-\delta)/2}\|\nabla^2\rho^\frac{\delta-1}{2}\|_{L^{2}}\|u\|_{L^\infty}\\
&\quad+\|\rho\|_{L^\infty}^{(1-\delta)/2}\|\nabla\rho^\frac{\delta-1}{2}\|_{L^{6}}\|\nabla u\|_{L^{3}}+\|\nabla\mathrm{div}u\|_{L^2}\bigr)\leq C\|\nabla^2 u\|_{L^2}+C.
\end{split}
\end{equation}

Then, following the similar estimates as \eqref{34}, one derives from \eqref{4.1}, \eqref{qa76}, \eqref{4.7}, and  \eqref{4.22}-\eqref{qaz49} that 
\begin{equation}
\begin{split}\label{lm99}
&\frac{1}{2}\frac{d}{dt}\|\rho^\frac{1-\delta}{2}u_t\|_{L^2}^2+\mu\|\mathrm{curl} u_t\|_{L^2}^2
+(2\mu+\lambda)\|\mathrm{div}u_t\|_{L^2}^2 +\mu\int_{\partial\Omega}u_t\cdot A\cdot u_tdS \\
&\leq C\int\Big(\rho^{1-\delta}|\nabla u||u_t|^2+\rho^{1-\delta}|\nabla^2 u||u|^2|u_t|+\rho^{1-\delta}|u||\nabla u|^2|u_t|
+\rho^{1-\delta}|u|^2|\nabla u||\nabla u_t|\\
&\quad+\rho^{1-\delta}|u||\nabla u_t||u_t|+|(\rho^{\gamma-\delta})_t||\mathrm{div}u_t|
+|(\nabla\log\rho)_t||\nabla u||u_t|+|\nabla\log\rho||\nabla u_t||u_t|\Big)dx\\
&\leq C\|\rho^\frac{1-\delta}{2}\|_{L^\infty}\|\nabla u\|_{L^3}\|\rho^\frac{1-\delta}{2}u_t\|_{L^2}\|u_t\|_{L^6}+
C\|\rho^{1-\delta}\|_{L^\infty}\|\nabla^2 u\|_{L^2}\|u\|_{L^6}^2\|u_t\|_{L^6}\\
&\quad+C\|\rho^{1-\delta}\|_{L^\infty}\|u\|_{L^6}\|\nabla u\|_{L^3}^2\|u_t\|_{L^6}
+C\|\rho^{1-\delta}\|_{L^\infty}\|u\|_{L^6}^2\|\nabla u\|_{L^6}\|\nabla u_t\|_{L^2}\\
&\quad+C\|\rho^\frac{3(1-\delta)}{4}\|_{L^\infty}\|u\|_{L^6}\|\nabla u_t\|_{L^2}\|\rho^\frac{1-\delta}{2}u_t\|_{L^2}^{1/2}\|u_t\|_{L^6}^{1/2}
+C\|(\rho^{\gamma-\delta})_t\|_{L^2}\|\mathrm{div}u_t\|_{L^2}\\
&\quad+C\|(\nabla\log\rho)_t\|_{L^2}\|\nabla u\|_{L^3}\|u_t\|_{L^6}\\
&\quad+C\|\rho^\frac{1-\delta}4\|_{L^\infty}\|\nabla\rho^\frac{\delta-1}{2}\|_{L^6}\|\nabla u_t\|_{L^2}\|\rho^\frac{1-\delta}{2}u_t\|_{L^2}^{1/2}\|u_t\|_{L^6}^{1/2}\\
&\leq \epsilon\|\nabla u_t\|_{L^2}^2+C(1+\|\nabla^2 u\|_{L^2}^2)\|\rho^\frac{1-\delta}{2}u_t\|_{L^2}^2.
\end{split}
\end{equation}
Using the trace theorem, the boundary term in \eqref{lm99} can be governed as
\begin{align}\label{trace31}
\int_{\partial\Omega}u_t\cdot A\cdot u_tdS\leq C\left\||u_t|^2\right\|_{W^{1,1}(\Omega_0)}  &\leq C\|\rho^\frac{1-\delta}{2}u_t\|_{L^2}^2+\epsilon\|\nabla u_t\|_{L^2}^2.
\end{align}

Choosing $\epsilon$ suitably small in \eqref{lm99},  Gr\"onwall's inequality combined with \eqref{trace31} yiedls
	\begin{equation}\nonumber
		\sup_{0\leq t\leq T}\|\rho^\frac{1-\delta}{2}u_t\|_{L^2}^2+\int_0^T\|\nabla u_t\|_{L^2}^2dt\leq C, \label{4.28}
	\end{equation}
	due to the fact that
	\begin{equation}\nonumber
		\rho_0^\frac{1-\delta}{2}u_t(x,0)=-\rho_0^\frac{1-\delta}{2}u_0\cdot\nabla u_0-\frac{\gamma a}{\gamma-\frac{1+\delta}{2}}\nabla\rho_0^{\gamma-\frac{1+\delta}{2}}
		+\frac{2\delta}{\delta-1}\nabla\rho_0^\frac{\delta-1}{2}\mathcal{S}(u_0)+g_2\in L^2.
	\end{equation}
	%which comes from the compatibility condition \eqref{118}. 
    
    Finally, it follows from \eqref{4.22} that
	\begin{equation}\nonumber
		\sup_{0\leq t\leq T}\|\nabla u\|_{H^1}\leq C. \label{4.30}
	\end{equation}
	%Thus, Lemma \ref{l53} follows from \eqref{4.28} and \eqref{4.30} immediately.
The proof of Lemma \ref{l53} is completed.
\end{proof}

\begin{lemma}\label{l54}
	Under the condition \eqref{4.1}, it holds that
	\begin{equation}\label{4.322}
		\sup_{0\leq t\leq T}\|\nabla\rho^{\gamma-\frac{1+\delta}{2}}\|_{L^6\cap D^{1,2}}+\int_0^T\|\nabla^3u\|_{L^2}^2 dt\leq C.
	\end{equation}
\end{lemma}
\begin{proof}
Considering  the Lam\'e system \eqref{e38} with the boundary condition \eqref{ch1} and far filed behavior condition \eqref{7}, the standard $L^p$-estimates   combined with \eqref{4.1}, \eqref{qa76}, \eqref{4.77},  \eqref{4.188}, and  \eqref{4.22} yield  that
\begin{equation}
\begin{split}\label{4.39re}
\|\nabla^2u\|_{L^6}&\leq C\|\mathcal{L}u\|_{L^6}+C\|\nabla u\|_{L^2}\\
&\leq C(\|\rho\|_{L^\infty}^{1-\delta}\|u_t\|_{L^6}+\|\rho\|_{L^\infty}^{1-\delta}\|u\|_{L^\infty}\|\nabla u\|_{L^6} +\|\rho\|_{L^\infty}^\frac{1-\delta}{2}\|\nabla\rho^{\gamma-\frac{1+\delta}{2}}\|_{L^6}\\
&\quad+\|\rho\|_{L^\infty}^\frac{1-\delta}2\|\nabla\rho^\frac{\delta-1}{2}\|_{L^{6}}\|\nabla u\|_{L^\infty}+1)\\
&\leq C(1+\|\nabla u_t\|_{L^2}+\|\nabla\rho^{\gamma-\frac{1+\delta}{2}}\|_{L^6})+\frac{1}{2}\|\nabla^2u\|_{L^6}\\
&\leq C(1+\|\nabla u_t\|_{L^2}+\|\nabla\rho^{\gamma-\frac{1+\delta}{2}}\|_{L^6}).
\end{split}
\end{equation}

Next, multiply \eqref{4.14} by $3|\nabla\phi|^4$, one gets
\begin{equation}
\begin{split}\nonumber
\frac{1}{2}\partial_t|\nabla\varphi|^6=&-3|\nabla\phi|^4\nabla\varphi\cdot\mathcal{D}(u)\cdot\nabla\varphi
-u\cdot\nabla(|\nabla\varphi|^6)-3\zeta|\nabla\varphi|^6\mathrm{div}u-3|\nabla\phi|^4\zeta\varphi\nabla\varphi\cdot\nabla \mathrm{div}u.
\end{split}
\end{equation}
Integrating the above equality over $\Omega$ and using \eqref{4.39re}, it holds that
\begin{equation}\label{4.15qq}
\begin{split}
\frac{1}{2}\frac{d}{dt}\|\nabla\varphi\|_{L^6}^2&\leq C\|\mathcal{D}(u)\|_{L^\infty}\|\nabla\varphi\|_{L^6}^2
+C\|\phi\|_{L^\infty}\|\nabla\phi\|_{L^6}\|\nabla^2 u\|_{L^6}\\
&\leq C(1+\|\mathcal{D}(u)\|_{L^\infty})\|\nabla\varphi\|_{L^6}^2
+C\|\nabla u_t\|_{L^2}^2,
\end{split}
\end{equation}
which together with Gr\"onwall's inequality,  \eqref{4.1}, and \eqref{4.188} implies
$$
\sup_{0\leq t\leq T}\|\nabla\rho^{\gamma-\frac{1+\delta}{2}}\|_{L^6}\leq C.
$$

Then, similar to \eqref{4.39re}, the standard $L^2$-estimates combining with \eqref{4.1}, \eqref{qa76}, \eqref{4.77},  \eqref{4.188}, and \eqref{4.22} yield  that
\begin{equation}
\begin{split}\label{4.39zz}
&\|\nabla^3u\|_{L^2}\leq C\|\mathcal{L}u\|_{H^1}+C\|\nabla u\|_{L^2}\\
&\leq C(\|\rho^{1-\delta}\|_{L^\infty}\|\nabla\rho^\frac{\delta-1}{2}\|_{L^{6}}\|\rho^\frac{1-\delta}{2}u_t\|_{L^{3}}
+\|\rho^{1-\delta}\|_{L^\infty}\|\nabla u_t\|_{L^2}\\
&\quad+\|\rho^{1-\delta}\|_{L^\infty}^{3/2}\|\nabla\rho^\frac{\delta-1}{2}\|_{L^\infty}\|u\|_{L^6}\|\nabla u\|_{L^3}
+\|\rho^{1-\delta}\|_{L^\infty}\|\nabla u\|_{L^3}\|\nabla u\|_{L^6}\\
&\quad+\|\rho^{1-\delta}\|_{L^\infty}\|u\|_{L^\infty}\|\nabla^2u\|_{L^2} +\|\rho^{1-\delta}\|_{L^\infty}^{3/2}\|\nabla\rho^\frac{\delta-1}{2}\|_{L^\infty}\|\nabla\rho^{\gamma-\frac{1+\delta}{2}}\|_{L^2}\\
&\quad+\|\rho^{1-\delta}\|_{L^\infty}^{1/2}\|\nabla^2\rho^{\gamma-\frac{1+\delta}{2}}\|_{L^2}
+\|\rho^\frac{1-\delta}{2}\|_{L^\infty}\|\nabla^2\rho^\frac{\delta-1}{2}\|_{L^{2}}\|\mathcal{S}(u)\|_{L^{\infty}}\\
&\quad+\|\rho^{1-\delta}\|_{L^\infty}\|\nabla\rho^\frac{\delta-1}{2}\|_{L^{6}}^2\|\mathcal{S}(u)\|_{L^{6}}
+\|\rho^\frac{1-\delta}{2}\|_{L^\infty}\|\nabla\rho^\frac{\delta-1}{2}\|_{L^{6}}\|\nabla^2u\|_{L^{3}}+1)\\
&\leq C(1+\|\nabla u_t\|_{L^2}+\|\nabla^2\rho^{\gamma-\frac{1+\delta}{2}}\|_{L^2})+\frac{1}{2}\|\nabla^3u\|_{L^2}\\
&\leq C(1+\|\nabla u_t\|_{L^2}+\|\nabla^2\rho^{\gamma-\frac{1+\delta}{2}}\|_{L^2}).
\end{split}
\end{equation}

Finally, similar to \eqref{e340}, one gets from \eqref{4.1} and Lemmas \ref{l51}-\ref{l53} that
\begin{equation}
\begin{split}\nonumber
\frac{d}{dt}\|\nabla^2\varphi\|_{L^2}^2
\leq& C\int|\nabla^2\varphi\cdot\nabla u\cdot\nabla^2\varphi|+|\nabla^2\varphi|^2|\mathrm{div} u|
+|\nabla\varphi||\nabla^2\varphi||\nabla^2u|+|\varphi||\nabla^2\varphi||\nabla^3u|\Big)dx\\ \leq&C(\|\mathcal{D}(u)\|_{L^\infty}\|\nabla^2\varphi\|_{L^2}^2+\|\nabla\varphi\|_{L^6}\|\nabla^2\varphi\|_{L^2}\|\nabla^2u\|_{L^3}\\
&+\|\varphi\|_{L^\infty}\|\nabla^2\varphi\|_{L^2}\|\nabla^3u\|_{L^2})\\ \leq&C(\|\mathcal{D}(u)\|_{L^\infty}\|\nabla^2\varphi\|_{L^2}^2+\|\nabla^2\varphi\|_{L^2}\|\nabla^3u\|_{L^2}^{1/2}+\|\nabla^2\varphi\|_{L^2}\|\nabla^3u\|_{L^2})\\
\leq&C(\|\mathcal{D}(u)\|_{L^\infty}+1)\|\nabla^2\varphi\|_{L^2}^2+C\|\nabla u_t\|_{L^2}^2.
\end{split}
\end{equation}
This, together with Gr\"onwall's inequality, \eqref{4.1}, \eqref{4.188},  and \eqref{4.39zz} implies \eqref{4.322} and completes the proof of Lemma \ref{l54}.
\end{proof}

Now, it is enough to extend the strong solutions of $(\rho, u)$ beyond $t > T^\ast$. In fact, Lemmas \ref{l52}-\ref{l54} imply that  $(\rho,u)|_{t=T^\ast} =
\lim_{t\rightarrow T^\ast} (\rho, u)$ satisfies the conditions imposed on the initial data \eqref{118} at the time $t = T^\ast$ .
Furthermore, it holds
\begin{equation}\nonumber
	\mathcal{L}(u)|_{t=T^\ast}=\lim_{t\rightarrow T^\ast}\big(\rho^{1-\delta}u_t+\rho^{1-\delta}u\cdot\nabla u+\frac{\gamma a}{\gamma-\delta}\nabla\rho^{\gamma-\delta}
	-\delta\nabla\log\rho\cdot\mathcal{S}(u)\big)\triangleq\rho^\frac{1-\delta}{2}g|_{t=T^\ast}
\end{equation}
with $g|_{t=T^\ast}\in L^2$. Thus, $(\rho, u)|_{t=T^\ast}$  also satisfies  \eqref{118}. Therefore, we can take $(\rho, u)|_{t=T^\ast}$
as the initial data and apply the local existence theorem to extend the local strong solution beyond $T^\ast$. This contradicts the assumption on $T^\ast$ and thus prove Theorem \ref{T2}.

\section*{Conflict-of-interest statement}
All authors declare that they have no conflicts of interest.

\section*{Acknowledgments}
This work is partially supported by the National Natural Science Foundation of China
(No. 12371219), and the Double-Thousand Plan of Jiangxi Province (No. jxsq2023201115).
%and the Academic and Technical Leaders Training Plan of Jiangxi Province (20212BCJ23027).

\end{document}